\documentclass[11pt]{amsart}
\usepackage[T1]{fontenc}
\usepackage{txfonts}
\usepackage{float}
\usepackage{amssymb,verbatim}
\usepackage[all]{xy}
\usepackage{subfig}
\usepackage{tikz}
\usepackage{color}
\usepackage{a4wide}
\usepackage{mathrsfs}
\usepackage{longtable}
\usepackage{hyperref}
\usepackage{wrapfig}
\usetikzlibrary{arrows}
\DeclareMathAlphabet{\mathpzc}{OT1}{pzc}{m}{it}

\newcommand{\R}{\mathbb{R}}
\newcommand{\C}{\mathbb{C}}
\newcommand\Z{\mathbb{Z}}
\newcommand{\N}{\mathbb{N}}
\newcommand{\Q}{\mathbb{Q}}

\newcommand{\Mb}{\mathbf{M}}
\newcommand{\Pb}{\mathbb{P}}

\newcommand{\Hbb}{\mathbb{H}}

\newcommand{\fraka}{\mathfrak{a}}

\newcommand{\Mcal}{\mathcal{M}}

\newcommand{\Pcal}{\mathcal{P}}
\newcommand{\Qcal}{\mathcal{Q}}

\newcommand{\Xcal}{\mathcal{X}}

\newcommand{\Bc}{\mathcal{B}}

\newcommand{\Oc}{\mathcal{O}}
\newcommand{\Pc}{\mathcal{P}}
\newcommand{\Qc}{\mathcal{Q}}

\newcommand{\id}{\mathrm{id}}

\newcommand{\SL}{{\rm SL}}
\newcommand{\GL}{{\rm GL}}

\newcommand{\Aa}{\textrm{Area}}

\newcommand{\vol}{{\rm vol}}

\newcommand{\End}{\mathrm{End}}

\renewcommand{\Im}{\mathrm{Im}}
\renewcommand{\Re}{\mathrm{Re}}

\newcommand{\Prym}{\mathrm{Prym}}
\newcommand{\odd}{\mathrm{odd}} 

\newcommand{\ol}{\overline}

\newtheorem{Theorem}{Theorem}[section]
\newtheorem{Corollary}[Theorem]{Corollary}
\newtheorem{Lemma}[Theorem]{Lemma}
\newtheorem{Proposition}[Theorem]{Proposition}

\newtheorem{Claim}[Theorem]{Claim}

\theoremstyle{remark}
\newtheorem{Remark}[Theorem]{Remark}
\begin{document}
\title[Modular forms and Siegel-Veech constants]{Hilbert surfaces, modular forms, and Siegel-Veech constants}
\author{Duc-Manh Nguyen}
\address{Université de Tours, Université d'Orléans, CNRS, IDP, UMR 7013, Parc de Grandmont, 37200 Tours, France}
\email[D.-M.~Nguyen]{duc-manh.nguyen@univ-tours.fr}

\date{\today}
\begin{abstract}
	We give the values of the Siegel-Veech constants associated with saddle connections having distinct endpoints on translation surfaces in Prym eigenform loci in $\Omega \Mcal_3(2,2)^\odd$. In particular, we show that these constants are actually the same for all of these loci. 
	As a by-product, we show that the Euler characteristic of the Hilbert modular surfaces which parametrize Abelian surfaces with $(1,2)$-polarization  admitting a real multiplication and the Euler characteristic of their product locus are related by a simple formula. 
	For principally polarized Abelian surfaces, a similar phenomenon  has been observed by Bainbridge. 
\end{abstract}

\maketitle

\section{Introduction}
\subsection{Statements of the results}\label{subsec:state:results}
Let $D$ be a natural number such that $D>9, \; D \equiv 0,1,4 \; [8]$, and $D$ is not a square. Denote by $\Oc_D$ the real quadratic order of discriminant $D$. Recall that $\Omega \Mcal_3(2,2)^\odd$ is the component of odd spin Abelian differentials having two double zeros on Riemann surfaces of genus three (see \cite{KZ03}). The space of Prym eigenforms for  real multiplication by $\Oc_D$ in $\Omega \Mcal_3(2,2)^\odd$ is denoted by $\Omega E_D(2,2)^\odd$ (for the definition of Prym eigenforms, we refer to \cite{McM:prym, LN:components}). This locus is a closed sub-orbifold of complex dimension three of $\Omega \Mcal_3(2,2)^\odd$ which is invariant by the $\GL^+(2,\R)$-action. It is shown in \cite{LN:components} that $\Omega E_D(2,2)^\odd$ is connected if $D\equiv 0,4 \, [8]$, and has two components if $D\equiv 1 \, [8]$ ($\Omega E_D(2,2)^\odd$ is empty if $D\equiv 5 \, [8]$).  

For $k=1,2,3$, let $c_k(D)$ denote the Siegel-Veech constant associated with saddle connections with multiplicity $k$ that join the two zeros of the Abelian differentials in $\Omega E_D(2,2)^\odd$ (in the case $D\equiv 1 \, [8]$, the Siegel-Veech constants of the two components of $\Omega EçD(2,2)^\odd$ are the same (cf. \cite{Ng:25})). The main aim of this paper is to prove the following

\begin{Theorem}\label{th:SV:const}
	For all $D\in \N, \; D >9, D \equiv 0,1,4 \, [8]$, and $D$ is not a square, we have
	$$
	c_1(D) = \frac{25}{9}, \quad c_2(D)= 3, \quad c_3(D) = \frac{2}{9}.
	$$
\end{Theorem}
The loci $\Omega E_D(2,2)^\odd$ are all contained in the locus  $\tilde{\Qc}(4,-1^4)$ of orienting double covers of quadratic differentials in the stratum $\Qc(4, -1^4)$.
Since the loci $\Omega E_D(2,2)^\odd$ are not contained in any proper invariant sub-orbifold of $\tilde{\Qcal}(4,-1^4)$ (see \cite{AN16})
it follows from the results of Eskin-Mirzakhani~\cite{EM18} and Eskin-Mirzakhani-Mohammadi~\cite{EMM13} that  as $D\to +\infty$, $c_k(D)$ converges to the corresponding Siegel-Veech constant of $\tilde{\Qc}(4,-1^4)$ that we denote by $\tilde{c}_k(4,-1^4)$. In Appendix~\ref{sec:SV:quad:cov}, we will give the proof that $\tilde{c}_1(4,-1^4)=25/9, \;, \tilde{c}_2(4,-1^4)= 3, \; \tilde{c}_3(4,-1^4)=2/9$ using the method of Eskin-Masur-Zorich~\cite{EMZ03}. 

\medskip

Hilbert modular surfaces are moduli spaces of Abelian surfaces admitting a real multiplication by some quadratic order  (see \textsection~\ref{sec:Hilbert:mod:surf} for more details). Let $X_D$ be the Hilbert  surface parametrizing principally polarized Abelian surfaces admitting a  multiplication by $\Oc_D$.  In this paper, we will be interested in Abelian surfaces with polarization of type $(1,2)$. The corresponding Hilbert surface  will be denoted by $X'_D$. Note that $X'_D=\varnothing$ if $D\equiv 5 \; [8]$ (see \cite{LN:components, Mo14}).

Let $P_D$ (resp. $P'_D$) denote the product locus of  $X_D$ (resp. of $X'_D$) that is the locus of Abelian surfaces $A \simeq E_1\times E_2$, where $E_1, E_2$ are elliptic curves (see \textsection\ref{subsec:prod:locus} for more details). 
The loci $P_D$ and $P'_D$ are finite covers of the modular curve $\Hbb/\SL(2,\Z)$. 
In \cite{Bai:GT}, Bainbridge showed that for all non-square discriminant $D >4$, we have
$$
\chi(P_D)=-\frac{5}{2}\cdot\chi(X_D)
$$
where $\chi(.)$ denotes the Euler characteristic. As a by product of the proof of Theorem \ref{th:SV:const}, we will show
\begin{Theorem}\label{th:eul:char:PD:p:XD:p}
	For all $D\in \N, D\equiv 0,1,4 \, [8]$, $D>9$ and $D$ is not a square, we have
	\begin{equation}\label{eq:eul:char:PD:p:XD:p}
		\chi(P'_D) =\left\{ 
		\begin{array}{ll}
			-\chi(X'_D) & \hbox{ if $D\equiv 1 \, [8]$} \\
			-\chi(X'_D)/2 & \hbox{ if $4 \, | \, D$}.
		\end{array} 
		\right.
	\end{equation}
	
\end{Theorem} 
\subsection{Ideas of the proof}\label{subsec:strategy}
Let $\Omega E_D(2)$ and $\Omega E_D(4)$ denote the Prym eigenfom loci with discriminant $D$ in the strata $\Omega \Mcal_2(2)$ and $\Omega \Mcal_3(4)$ respectively. Each of these loci is a finite union of closed $\GL^+(2,\R)$-orbits  (see \cite{McM:prym}). Let $W_D(2)$ (resp. $W_D(4)$) denote the projectivization of $\Omega E_D(2)$ (resp. of $\Omega E_D(4)$), that is $W_D(2)\simeq \Omega E_D(2)/\C^*$ (resp. $W_D(4)\simeq \Omega E_D(4)/\C^*$). Note that $W_D(2)$ (resp. $W_D(4)$) is isomorphic to a finite union of Teichm\"uller curves.    

To our purpose, we will consider Abelian differentials on disconnected Riemann surfaces. In \cite{Ng:25} (see also \cite{LN:components}), we introduced the notion of Prym eigenforms on triples of tori. The space of triples of tori that are Prym eigenforms for the quadratic order $\Oc_D$ is denoted by $\Omega E_D(0^3)$, and its projectivization  $\Omega E_D(0^3)/\C^*$  by $W_D(0^3)$. It can be shown that $W_D(0^3)$ is a (non-connected) finite cover of the modular curve $\Hbb/\SL(2,\Z)$. The following result was proved in \cite{Ng:25}.
\begin{Theorem}\label{th:Siegel:Veech}
	Let $D \equiv 0, 1, 4 \;  [8], \, D >9$, be a non-square discriminant. 	
	\begin{itemize} 
		\item[$\bullet$] If $4 \, | \, D$,  then we have
		\begin{align*}
			c_1(D) & = \frac{15\chi(W_D(4))}{\chi(W_D(2))+b_D \chi(W_{D/4}(2))+9\chi(W_D(0^3))}\\
			c_2(D) &= \frac{9\left(\chi(W_{D}(2)) + b_D\chi(W_{D/4}(2)) \right)}{\chi(W_D(2))+b_D \chi(W_{D/4}(2))+9\chi(W_D(0^3))} \\
			c_3(D) & = \frac{3\chi(W_D(0^3))}{\chi(W_D(2))+b_D \chi(W_{D/4}(2))+9\chi(W_D(0^3))}
		\end{align*}
		with
		$$
		b_D= \left\{
		\begin{array}{ll}
			0 & \text{ if } D/4 \equiv 2,3 \, [4] \\
			4 & \text{ if } D/4 \equiv 0 \, [4] \\
			3 & \text{ if } D/4 \equiv 1 \, [8] \\
			5 & \text{ if } D/4 \equiv 5 \, [8]. \\
		\end{array}
		\right.
		$$
		\item[$\bullet$] If $D\equiv 1 \; [8]$, then
		\begin{align*}
		c_1(D) & = \frac{15\chi(W_{D}(4))}{2\chi(W_D(2))+9\chi(W_D(0^3))}\\
		c_2(D) &= \frac{18\chi(W_{D}(2))}{2\chi(W_{D}(2))+9\chi(W_D(0^3))} \\
		c_3(D) & = \frac{3\chi(W_{D}(0^3))}{2\chi(W_{D}(2))+9\chi(W_D(0^3))}.
		\end{align*}
	\end{itemize}	
\end{Theorem}
 For any $D \in \N$, $D >9$ non-square and $D\equiv 0,1,4 \, [8]$, by a result of M\"oller~\cite{Mo14}, we have that
 $$
 \chi(W_D(4))=\left\{
 \begin{array}{cl}
 {\displaystyle -\frac{5}{2} \cdot\chi(X'_D)} & \hbox{ if $D \equiv 0 \, [4]$} \\
 {\displaystyle -5\cdot\chi(X'_D)} & \hbox{ if $D \equiv 1 \, [8]$}.	 
 \end{array}
 \right.
 $$
 It was shown by Bainbridge in \cite{Bai:GT} that we have 
 $$
 \chi(W_D(2))=-\frac{9}{2}\cdot\chi(X_D).
 $$
 In the same paper, Bainbridge also gave an explicit formula computing $\chi(X_D)$ generalizing a classical result of Siegel.  In \cite{Mo14}, M\"oller proved an explicit formula relating  $\chi(X'_D)$ to $\chi(X_D)$ (cf. Proposition~\ref{prop:eul:char:XD:p}). Combining these results,  we will show that 
 $$
 \chi(W_D(2)) +b_D\chi(W_{D/4}(2))= -\frac{9}{2}\chi(X'_D).
 $$  
To prove Theorem~\ref{th:SV:const}, it remains to show that we have
$$
\chi(W_D(0^3))= \left\{ 
\begin{array}{rl}
	-\chi(X'_D) & \hbox{ if $D\equiv 0 \, [4]$} \\
	-2\chi(X'_D) & \hbox{ if $D \equiv 1 \, [8]$}
\end{array} 
\right.
$$
This is achieved by relating $\chi(W_D(0^3))$ to  coefficients of the Fourier expansion of some modular forms. 

\subsection*{Acknowledgement} The author is grateful to Aurel Page for the very helpful discussions.

\section{Consequences of relations of modular forms}\label{sec:conseq:rel:mod:forms}
Recall that for any positive integers $m,n$,  
$$
\sigma_m(n):=\sum_{d \, | \, n } d^m.
$$
By convention $\sigma_m(0)=\frac{1}{2}\cdot\zeta_\Q(-m)$ and $\sigma_m(n)=0$ if $n<0$. Our goal in this section is to show the following
\begin{Theorem}\label{th:rel:sum:sigma:1}
Let $D \in \N$ be a non-square discriminant. 
\begin{itemize}
	\item[(i)] If $D \equiv 0 \, [4]$ then we have 
\begin{equation}\label{eq:rel:sum:sigma:1:D:4}
\sum_{\substack{e^2 \equiv D \, [8]}}\sigma_1(\frac{D-e^2}{8}) = \frac{1}{5}\cdot \sum_{\substack{e \; {\rm even}}}\sigma_1(\frac{D-e^2}{4}) +\frac{4}{5}\cdot \sum_{\substack{e^2 \equiv D/4 \, [4]}}\sigma_1(\frac{D/4-e^2}{4}).
\end{equation} 

\item[(ii)] If $D \equiv 1 \, [8]$ then we have 
\begin{equation}\label{eq:rel:sum:sigma:1:D:1:8}
	\sum_{\substack{e \; \odd}}\sigma_1(\frac{D-e^2}{8}) = \frac{2}{5}\cdot \sum_{\substack{e \;  \odd}}\sigma_1(\frac{D-e^2}{4}). 
\end{equation} 
\end{itemize}
\end{Theorem}

\begin{Remark}\label{rk:rel:sum:sigma:1:D:4}
	In the cases $D/4 \equiv 2,3 \, [4]$, $e^2 \not\equiv D/4 \; [4]$ for any integer $e$. Thus \eqref{eq:rel:sum:sigma:1:D:4} becomes
	$$
	\sum_{\substack{e^2 \equiv D \, [8]}}\sigma_1(\frac{D-e^2}{8}) = \frac{1}{5}\cdot \sum_{\substack{e \; {\rm even}}}\sigma_1(\frac{D-e^2}{4}).
	$$
\end{Remark}

\subsection{Basic definitions and properties of some modular functions}\label{subsec:basic:mod:forms}
We first recall some classical notation and conventions.
\begin{enumerate}
	\item For all integers $a,b$, with $b \neq 0$, ${\displaystyle \left( \frac{a}{b}\right)}$ is the Kronecker symbol which satisfies the following properties (see \cite{CO76})
	\begin{itemize}
		\item Given $a$, the function $b \mapsto \left(\frac{a}{b}\right)$ is completely multiplicative.
		
		\item ${\displaystyle \left( \frac{a}{-1}\right)=-1}$ if  $c < 0$ and ${\displaystyle \left( \frac{a}{-1}\right)=1}$ if  $c \geq  0$.
		
		\item   ${\displaystyle \left( \frac{a}{2}\right)=(-1)^{(a^2-1)/8}}$ if $a$ is odd.
		
		\item ${\displaystyle \left( \frac{a}{p}\right)}$ is the Legendre symbol if $p$ is an odd prime, and $\gcd(a,p)=1$.
		
		\item ${\displaystyle \left( \frac{a}{b}\right)=0}$ if $\gcd(a,b) \neq 1$.
	\end{itemize}

	\item For all integer $N\geq 1$, 
	$$
	\Gamma_0(N)=\{\left(\begin{array}{cc}  a & b \\ c & d \end{array}\right) \in \SL(2, \Z), \; c\equiv 0 \, [N]\}.
	$$ 
	\item For all $k\in \frac{1}{2}\cdot\Z$ and $N\in \N, N>0$, $\Mb_k(\Gamma_0(N))$ is the space of integral modular forms of weight $k$ and level $N$. Elements of  $\Mb_k(\Gamma_0(N))$ are functions $f: \Hbb \to \C$, where $\Hbb:=\{z\in \C, \; \Im(z) >0\}$, satisfying the followings 
	\begin{itemize}
		\item[a)] $f$ is holomorphic on $\Hbb$,
		
		\item[b)] If $k\in \Z$, then and for all $z\in \Hbb$ and for all $\gamma=\left(\begin{array}{cc} a & b \\ c & d \end{array}\right) \in \Gamma_0(N)$, we have
		$$
		f(\gamma\cdot z):= f(\frac{az+b}{cz+d}) = (cz+d)^{k}f(z) \quad \hbox{ for $k\in \Z$,}
		$$
		and
		$$
		f(\gamma\cdot z) = \left(\frac{c}{d}\right)\cdot\left(\frac{-4}{d}\right)^{-k}(cz+d)^{k}f(z) \quad \hbox{for $k \in \frac{1}{2}+\Z$}.
		$$
		\item[c)] $f$ is holomorphic at the cusps of $\Hbb/\Gamma_0(N)$.
	\end{itemize}
	
	\item For all $z\in \Hbb$,
	$$
	G_2(z)=\frac{-1}{24}+\sum_{n=1}\sigma_1(n)q^n
	$$
	where $q=e^{2\pi\imath z}$. 
	The function $G_2$ satisfies the following property: for all $\gamma=\left(\begin{smallmatrix} a & b \\ c & d \end{smallmatrix}\right) \in \SL(2,\Z)$, we have
	$$
	G_2(\gamma\cdot z)=(cz+d)^2G_2(z)-\frac{c(cz+d)}{4\pi\imath}.
	$$
	
	\item For all $z\in \Hbb$,
	$$
	\theta(z)=\sum_{n\in\Z}q^{n^2}=1+2\sum_{n=1}^\infty q^{n^2}.
	$$
	For all $\gamma=\left(\begin{smallmatrix} a & b \\ c & d \end{smallmatrix}\right) \in \Gamma_0(4)$, we have
	$$
	\theta(\gamma\cdot z)=\left(\frac{c}{d} \right)\left(\frac{-4}{d}\right)^{-1/2}(cz+d)^{1/2}\theta(z).
	$$
	As a consequence
	$$
	\theta'(\gamma\cdot z)=\left(\frac{c}{d} \right)\left(\frac{-4}{d}\right)^{-1/2}\left( \frac{c}{2}(cz+d)^{3/2}\theta(z) +(cz+d)^{5/2}\theta'(z)\right).
	$$	
\end{enumerate}

\subsection{Relations of modular forms}\label{subsec:rel:mod:forms}
Define
\[
f_{1}(z) := G_2(z)\theta(z)+\frac{1}{2\pi\imath}\theta'(z), \quad  
	f_{2}(z):= G_2(2z)\theta(z)+\frac{1}{4\pi\imath}\theta'(z),  \quad 
	f_{4}(z):= G_2(4z)\theta(z)+\frac{1}{8\pi\imath}\theta'(z) 
\]

\begin{Remark}\label{rk:f:4:1:by:Cohen}
	The function $f_{4}$ has been introduced in \cite{C75}.
\end{Remark}
Let
\[
f_{1}(z) =\sum_{n=0}^\infty a_{1}(n)q^n, \quad 
f_{2}(z) =\sum_{n=0}^\infty a_{2}(n)q^n, \quad 
f_{4}(z) =\sum_{n=0}^\infty a_{4}(n)q^n 
\]
be the Fourier expansions of $f_{1}, f_{2}, f_{4}$ respectively. 
The following lemma follows immediately from the definition of $f_{1}, f_{2},f_{4}$. 
\begin{Lemma}\label{lm:coeff:Fourier:f:m:n}
Let $D$ be a positive integer. If $D$ is not a square then we have
\begin{align*}
	a_{1}(D) & =\sigma_1(D)+2\sum_{1\leq e  < \sqrt{D}}\sigma_1(D-e^2)= \sum_{e\in\Z}\sigma_1(D-e^2),\\
	a_{2}(D) & = \sum_{e\equiv D \, [2]}\sigma_1(\frac{D-e^2}{2}),\\
	a_{4}(D) & = \sum_{e^2\equiv D \, [4]}\sigma_1(\frac{D-e^2}{4}).
\end{align*}
In particular we have $a_{4}(D)=0$ if $D\equiv 2,3 \, [4]$.

If $D=d^2$, $d\in\N$, then we have
\begin{align*}
	a_{1}(d^2) &= \sum_{e \in \Z} \sigma_1(d^2-e^2) +2d^2,\\
	a_{2}(d^2) &= \sum_{e \equiv d \, [2]} \sigma_1(\frac{d^2-e^2}{2}) +d^2,\\
	a_{4}(d^2) &= \sum_{e \equiv d \, [2]} \sigma_1(\frac{d^2-e^2}{4}) + \frac{d^2}{2}.\\
\end{align*}
\end{Lemma}

\begin{Proposition}\label{prop:rel:f:1:2:4}
	The functions $f_{1}, f_{2}, f_{4}$ all belong to the space $\Mb_{5/2}(\Gamma_0(4))$ of integral modular forms of weight $5/2$ with respect to $\Gamma_0(4)$. More over, we have the following relation
	\begin{equation}\label{eq:rel:in:M:5:2}
		f_{1}-5f_{2}+4f_{4}=0.
	\end{equation}
	As a consequence, we have
	\begin{equation}\label{eq:rel:coeffs:in:Gam:0:4}
		a_{1}(n)-5a_{2}(n)+4a_{4}(n)=0,
	\end{equation}
	for all $n\in\N$. 
\end{Proposition}
\begin{proof}
	That $f_{1}, f_{2}, f_{4}$ belong to $\Mb_{5/2}(\Gamma_0(4))$ follows immediately from the properties of $G_2$ and $\theta$. \\ 
	Since $\dim \Mb_{5/2}(\Gamma_0(4))=2$, there must be a linear relation between $f_{1}, f_{2}$, and $f_{4}$. Plugging the values of $a_{1}(n), a_{2}(n), a_{4}(n)$ for some $n$ allows us to calculate explicitly the coefficients of this relation.	
\end{proof}

The following lemma follows from elementary arguments.
\begin{Lemma}\label{lm:coeffs:f:1:n:f:4}
	For all $n\in\N, n \geq 1$, we have
	$$
	a_{1}(n)=a_{4}(4n).
	$$
	As a consequence
	$$
	g_1(z):=f_{4}(z)-f_{1}(4z) = \sum_{n \equiv 1 \, [4], \, n >0}a_{4}(n)q^n \in \Mb_{5/2}(\Gamma_0(16)).
	$$
\end{Lemma}

%

Let $\psi: (\Z/8\Z)^*\simeq (\Z/2\Z)^2 \to \{\pm 1\}$ be a character such that $\psi(5)=-1$. 
The conductor of $\psi$ is $8$, and  $\psi^2$ is trivial.
Define
$$
g_2(z)=g_{1\psi}(z) :=\sum_{n\equiv 1 \, [8], \, n >0} a_{4}(n)q^n - \sum_{n\equiv 5 \, [8], \, n >0} a_{4}(n)q^n.
$$
Then we have $g_2\in \Mb_{5/2}(\Gamma_0(64))$ (cf. \cite[Lem. 4.3.10]{Miyake}). Define
$$
g:=\frac{g_1+g_2}{2}=\sum_{n\equiv 1 \, [8], \, n>0} a_{4}(n)q^n  \in \Mb_{5/2}(\Gamma_0(64)).
$$

Let
$$
f_{8}(z):=G_2(8z)\theta(z)+\frac{1}{16\pi\imath}\theta'(z) =  \sum_{n=0}^\infty a_{8}(n)q^n.
$$
The following lemmas are elementary
\begin{Lemma}\label{lm:coeffs:f:8:1}
	We have $f_{8}\in \Mb_{5/2}(\Gamma_0(8))$. 
	Let $D$ be a positive integer. If $D$ is not a square then
	$$
	a_{8}(D)=\sum_{e^2 \equiv D \, [8]}\sigma_1(\frac{D-e^2}{8}).
	$$
	In particular, we have $a_{8}(n) \neq 0$ only if $n\equiv 0, 1, 4 \, [8]$. 
	
	If $D=d^2$, then we have
	$$
	a_{8}(d^2)=\sum_{e^2\equiv d^2 \, [8]}\sigma_1(\frac{d^2-e^2}{8})+\frac{d^2}{4}
	$$
\end{Lemma}

\begin{Lemma}\label{lm:coeffs:f:8:1:n:f:2:1}
	For all $n\in \N$, we have $a_{2}(n)=a_{8}(4n)$.
	As a consequence
	$$
	h(z):=f_{8}(z)-f_{2}(4z)=\sum_{n \equiv 1 \, [8], n>0}a_{8}(n)q^n \in \Mb_{5/2}(\Gamma_0(16)).
	$$
\end{Lemma}

\begin{Proposition}\label{prop:rel:coeff:f:8:n:f:4}
	We have  
	\begin{equation}\label{eq:rel:mod:forms:g:h}
	h = \frac{2}{5}g \in \Mb_{5/2}(\Gamma_0(64)).
    \end{equation}
	As a consequence
	\begin{equation}\label{eq:rel:a:8:n:a:4}
		a_{8}(n)=\frac{2}{5} a_{4}(n)
	\end{equation}
	for all $n >0, n\equiv 1 \, [8]$.
\end{Proposition}
\begin{proof}
It follows from \cite[Th. 4.2.7]{Miyake} that the curve $X_0(64):=\Hbb/\Gamma_0(64)$ has no elliptic points and $12$ cusps. We have 
$$
[\ol{\Gamma}(1):\ol{\Gamma}_0(64)]= [\Gamma(1):\Gamma_0(64)]= 64(1+\frac{1}{2})=96.
$$
Therefore $g(X_0(64))=3$ (cf. \cite[Th. 4.2.11]{Miyake}).
We consider $f:=(h-2/5\cdot g)^4$. By definition, $f$ belongs to $\Mb_{10}(\Gamma_0(64))$. Let $f(z)=\sum_{n\in \N} b(n)q^n$ be the Fourier expansion of  $f$. One readily checks that $a_8(n)=2/5\cdot a_4(n)$ for all $n \equiv 1 \; [8]$ and $n\leq 25$. Therefore, $b(n)=0$ for all $n\leq 100$.
Since $f \in \Mb_{10}(\Gamma_0(64))$, if $f\not\equiv  0$ then $f$ cannot vanish to the order $5(2\times 3- 2)+5\times 12 =80$ at $\infty$ (cf. \cite[Cor. 2.3.4]{Miyake}). Thus we can conclude that $f \equiv 0$ and $h=2/5 \cdot g$.
\end{proof}
\subsection{Proof of Theorem~\ref{th:rel:sum:sigma:1}}
\begin{proof}
	Let  $D \in \N$ be a non-square discriminant.
	Suppose that $D \equiv 0 \, [4]$. We then have 
	\begin{align*}
		a_1(D/4) & =\sum_{e \in \Z} \sigma_1(D/4-e^2) = \sum_{e\in \Z}\sigma_1(\frac{D-4e^2}{4})=\sum_{e \, {\rm even}}\sigma_1(\frac{D-e^2}{4})= a_4(D)\\
		a_2(D/4) & = \sum_{e \equiv D/4 \, [2]}\sigma_1(\frac{D/4-e^2}{2}) = \sum_{e \equiv D/4 \; [2]}\sigma(\frac{D- 4e^2}{8}) = \sum_{ e^2 \equiv D \, [8]} \sigma_1(\frac{D-e^2}{8})=a_8(D).
	\end{align*}
Thus \eqref{eq:rel:sum:sigma:1:D:4} follows from Proposition~\ref{prop:rel:f:1:2:4}.

If $D \equiv 1 \, [8]$ we obtain \eqref{eq:rel:sum:sigma:1:D:1:8} from Proposition~\ref{prop:rel:coeff:f:8:n:f:4}.
\end{proof}

\section{Real multiplication and Hilbert modular surfaces}\label{sec:Hilbert:mod:surf}
\subsection{Hilbert modular surfaces}\label{subsec:Hilbert:surf}
A real quadratic order is a ring  isomorphic to $\Z[x]/(x^2+bx+c)$, with $b,c\in \Z$ such that $D:=b^2-4c >0$. The number $D$ is called the discriminant of the order. 
A quadratic order is determined up to isomorphism  by its discriminant. For all $D \in \N, D\equiv 0,1 \; [4]$, we will denote by $\Oc_D$ the real quadratic order of discriminant $D$.

Let $A$ be an Abelian surface. We say that $A$ admits a real multiplication by $\Oc_D$ if there exists a faithful ring morphism $\rho: \Oc_D \to \End(A)$ such that
\begin{itemize}
	\item the image of $\rho$ consists of self-adjoint endomorphisms with respect to the polarization of $A$.
	
	\item $\rho$ is proper, meaning that if $f \in \End(A)$, and for some $n \in \Z\setminus\{0\}$, we have $nf\in \rho(\Oc_D)$, then $f \in \rho(\Oc_D)$.
\end{itemize}

For principally polarized Abelian surfaces and $D$ non-square, the set of pairs $(A,\rho)$ as above is parametrized by the Hilbert modular surface 
$$
X_D \simeq \Hbb\times\Hbb/\SL(\Oc_D\oplus\Oc^\vee_D),
$$ 
where $\Oc_D^\vee=\frac{1}{\sqrt{D}}\cdot \Oc_D$ is the {\em inverse different} of $\Oc_D$, and $\SL(\Oc_D\oplus\Oc_D^{\vee})$ is the subgroup of $\SL(2,\Q(\sqrt{D}))$ that preserves the lattice $\Oc_D\oplus\Oc_D^\vee$ (see \cite[\textsection 2]{Bai:GT} or \cite[\textsection 2]{McM:fol} for more details). Actually, we have
$$
\SL(\Oc_D\oplus \Oc_D^\vee) = \{\left(\begin{array}{cc} a & b \\ c & d \end{array} \right), \; a,d \in \Oc_D, \; b \in\left( \Oc_D^\vee\right)^{-1}, \; c \in \Oc_D^\vee, \; ad-bc=1 \}
$$
where $\left(\Oc_D^\vee\right)^{-1} = \sqrt{D}\cdot \Oc_D$. It can also be shown that $X_D\simeq \Hbb\times(-\Hbb)/\SL(2,\Oc_D)$.

Let $z_i, \; i=1,2$, be the complex coordinate on the two copies of $\Hbb$ in  $\Hbb\times\Hbb$. Let $x_i=\Re(z_i), \; y_i=\Im(z_i)$. Then the $2$-form 
\begin{equation}\label{eq:vol:form:Hyp}
\eta_i:=\frac{1}{2\pi}\cdot\frac{dx_i\wedge dy_i}{y_i^2}.
\end{equation}
descends to a $2$-form on $X_D$.
The Euler characteristic of $X_D$ can be computed by 
$$
\chi(X_D)=\int_{X_D}\eta_1\wedge \eta_2.
$$
A discriminant $D\in \N$ which cannot be written as $D=f^2D'$, where $f\in \N, \; f>1$, and $D'$ is also a discriminant (that is $D'\equiv 0,1 [4]$) is called  {\em fundamental}. For all fundamental discriminant $D>1$, Siegel \cite{Sie36} proved that
$$
\chi(X_D)=2\zeta_{K_D}(-1)
$$ 
where $\zeta_{K_D}$ is the Dedekind zeta function of the quadratic field $K_D:=\Q(\sqrt{D})$. 
In \cite{Bai:GT}, Bainbridge generalized the Siegel's result to all discriminants. 
\begin{Theorem}[Bainbridge]\label{th:Bainb:E:char:XD}
	Let $D_0>1$ be a fundamental discriminant. For all $f\in \N, f \geq 1$,  we have
	\begin{equation}\label{eq:eul:char:XD}
		\chi(X_{f^2D_0})=2\cdot f^3\cdot \zeta_{K_{D_0}}(-1)\cdot\sum_{r \, | \, f} \left( \frac{D_0}{r}\right)\frac{\mu(r)}{r^2} 
	\end{equation} 
	where $\left(\frac{*}{*}\right)$ is the Kronecker symbol and $\mu(.)$ is the M\"obius function.
\end{Theorem}

Combining \eqref{eq:eul:char:XD} and a result of  Cohen in \cite{C75}, one has
\begin{Corollary}[Bainbridge]\label{cor:sum:eul:char:XD}
	Let $D >1$ be a non-square discriminant. Let us write $D=f^2D_0$ where $D_0$ is a fundamental discriminant. Then we have
	\begin{equation}\label{eq:sum:eul:char:XD}
	\sum_{r \, | \, f}\chi(X_{r^2D_0}) =\frac{1}{30}\cdot\sum_{\substack{e \equiv D \, [2]}} \sigma_1\left(\frac{D-e^2}{4}\right)
	\end{equation}
\end{Corollary} 

\subsection{Product locus in $X_D$}\label{subsec:prod:locus}
Following \cite{Bai:GT}, we call the locus of points in $X_D$ which represent Abelian surfaces that are (polarized) product of elliptic curves the {\em product locus} of $X_D$ and denote it by $P_D$. Generic points in $X_D$ represents  Jacobians of smooth genus two curves. A point in $P_D$ represents the Jacobian of a singular curve having two irreducible components of genus one intersecting at a separating node. 
It is well known that $P_D$ is a union of modular curves \cite{Hir73, vdG88}. 
In \cite{Bai:GT}, Bainbridge proved the following
\begin{Theorem}[Bainbridge]\label{th:Bain:eul:char:PD}
	If the discriminant $D$ is not a square then we have 
	\begin{equation}\label{eq:eul:char:PD}
		\chi(P_D) = -\frac{5}{2}\cdot \chi(X_D).
	\end{equation}
\end{Theorem}
\begin{proof}[Sketch of proof]
	Let $\Omega\Mcal_1$ denote the space of pairs $(E,\omega)$ where $E$ is an elliptic curve and $\omega$ is a holomorphic one form on $E$. 
	Let $\Omega Q_D$ denote the subset of $\Omega \Mcal_{1}\times\Omega \Mcal_1$ consisting of elements $((E_1,\omega_1),(E_2,\omega_2))$ such that 
	\begin{itemize}
		\item the Abelian surface $E_1\times E_2$ admits a real multiplication by $\Oc_D$,
		
		\item the holomorphic $1$-form $\omega:=\omega_1+\omega_2$ on $E_1\times E_2$ is an eigenvector of the action of $\Oc_D$ on $\Omega^1(E_1\times E_2)$ (here by a slight abuse of notation, the pullback of $\omega_i$ to $E_1\times E_2$ is denoted again by $\omega_i$).
	\end{itemize}
Let $Q_D$ denote the quotient  $\Omega Q_D/\C^*$. We claim that
\begin{equation}\label{eq:eul:char:QD}
	\chi(Q_D)=-5\chi(X_D)
\end{equation}
Call a triple of integers $(e,\ell, m)$ such that 
$$
D = e^2 + 4\ell^2m \quad \ell >0, m >0, \quad \text{and} \quad \gcd(e,\ell)=1
$$
a {\em product of tori prototype} for the  discriminant $D$. 
In \cite{McM:spin}, McMullen showed that each connected component of $\Omega Q_D$ corresponds uniquely to a product of tori prototype $(e,\ell,m)$, and is isomorphic to $\GL(2,\R)/\Gamma_0(m)$.  It is  well known that 
$$
\chi(\Hbb/\Gamma_0(m))=-\frac{1}{6}\cdot m \cdot \prod_{\substack{p \, | \, m \\ p \; {\rm prime}}}\left(1+\frac{1}{p} \right)=:\psi(m),
$$
and we have
$$
\sum_{\substack{n=\ell^2 m \\ \ell, m \in \N}} \psi(m)=-\frac{1}{6}\sigma_1(n).
$$
Let us write $D=f^2D_0$, where $D_0$ is a fundamental discriminant. 
We have
\[
\sum_{r \, | \, f }\chi(Q_{r^2D_0}) = \sum_{\substack{ D=e^2+4\ell^2m \\ \ell, m  >0}}\psi(m) =-\frac{1}{6}\sum_{\substack{e \equiv D \, [2] }}\sigma_1(\frac{D-e^2}{4}).
\]
Combining with \eqref{eq:sum:eul:char:XD}, we get
\[
\sum_{r \, | \, f }\chi(Q_{r^2D}) = -5\cdot \sum_{r \, | \, f}\chi(X_{r^2D})
\]
and the claim follows.

Let $F: Q_D \to P_D$ be the map consisting of forgetting the eigenform $\omega$. For every $E_1\times E_2 \in P_D$,   $\Oc_D$ has two eigenforms $\omega, \omega'$ in $\Omega^1(E_1\times E_2)$ such that $\Omega^1(E_1\times E_2) \simeq \C\cdot\omega\oplus\C\cdot\omega'$. Thus $F$ is a two-to-one covering map, from which we get the desired conclusion. 
\end{proof} 

\subsection{Prym eigenforms and Abelian surfaces with $(1,2)$-polarization}
Consider an Abelian differential $(C,\omega) \in \Omega E_D(2,2)^\odd$.
By definition  $C$ is a smooth curve of genus $3$ admitting an involution $\tau$ which has $4$ fixed points. 
Denote by $\Omega^-(C,\tau)$ the eigenspace associated with the eigenvalue $-1$ of the action of $\tau^*$ on $\Omega(C)$.
Note that we have  $\dim_\C\Omega(C,\tau)^-=2$.
Let 
$$
H_1(X,\Z)^-:=\{c\in H_1(C,\Z), \; \tau_*c=-c\}.
$$
The Prym variety of $C$ with respect to $\tau$ is defined to be (see for instance \cite{McM:prym})
$$
\Prym(C,\tau):=\Omega^-(C,\tau)^*/H_1(X,\Z)^-.
$$    
The surface $\Prym(C,\tau)$ is a subvariety of the Jacobian $J(C)$ of $C$, and the  polarization of $J(C)$ induces a polarization   of type $(1,2)$ on $\Prym(C,\tau)$ (cf. \cite{Bar87, Mo14}).

Let $D\in \N$ be a discriminant.   
Let $X'_D$ be the moduli space of Abelian surfaces with $(1,2)$ polarization admitting real multiplication by $\Oc_D$. Then $X'_D$ is a complex surface constructed in a similar way as $X_D$. In particular, we have 
$$
X'_D=(\Hbb\times\Hbb)/\SL(\fraka\oplus\Oc^\vee_D),
$$
where $\fraka$ is a fractional  $\Oc_D$-ideal of norm $2$, and $\SL(\fraka\oplus\Oc^\vee_D)$ is the subgroup of $\SL(2,\Q(\sqrt{D}))$ preserving the  lattice $\fraka\oplus\Oc^\vee_D$ (see \cite[Prop.1.1]{Mo14} for more details).  Note that $X'_D$ is empty if $D\equiv 5 \; [8]$ (cf. \cite{LN:H4, Mo14}). We will also call $X'_D$ a Hilbert modular surface.  

Let $\eta_i, \; i=1,2$ be the $2$-forms on the two factors of $\Hbb\times\Hbb$ defined in \eqref{eq:vol:form:Hyp}. The Euler characteristic of $X'_D$ is then given by
\begin{equation}
\chi(X'_D)=\int_{X'_D}\eta_1\wedge\eta_2
\end{equation}
The following result was proved in \cite{Mo14}
\begin{Proposition}[M\"oller]\label{prop:eul:char:XD:p}
	Assume that $D\equiv 0,1,4 \; [8]$ and $D=f^2D_0$, where $D_0> 1$ is a fundamental discriminant. Then we have
	\begin{equation}\label{eq:eul:char:XD:p:ratio}
	\frac{\chi(X'_D)}{\chi(X_D)} =\left\{ 
	\begin{array}{ll}
		1 & \hbox{ if $2 \nmid f$},\\
		3/2 & \hbox{ if $2 \, | \, f$}.
	\end{array}
	\right.
     \end{equation} 
\end{Proposition} 

\subsection{Eigenforms on products of elliptic curves}\label{subsec:eigen:form:on:product}
Let $P'_D \subset X'_D$ denote the locus of Abelian surfaces $A$  that are isomorphic to a products of two elliptic curves $E_1\times E_2$ and satisfies the following: let $L_i$ be the line bundle defining the principal polarization on $E_i$, then the polarization on $A$ is given by $p_1^*L_1\otimes p_2^*L_2^{\otimes 2}$, where $p_i: A \to E_i$ is the natural projection. 
In more concrete terms: let $\langle.,.\rangle_i$ denote the intersection forms on $H_1(E_i,\Z)$. Then the polarization of $A$ is given the skew-symmetric bilinear form $\langle.,.\rangle:= \langle.,.\rangle_1+2\langle.,.\rangle_2$ on $H_1(A,\Z)\simeq H_1(E_1,\Z)\times H_1(E_2,\Z)$.  

A smooth curve  is called {\em bielliptic} if it admits a morphism of degree two onto an elliptic curve.  
Generic points in $X'_D$ correspond to Prym varieties of smooth genus three bielliptic curves. The bielliptic structure on such a curve is given by an involution  fixing four points. Call this involution  {\em Prym involution}. 
The notion of Prym involution can be extended to singular curves. In particular, let $C$ be a singular curve of genus three having four irreducible components: three of the components are elliptic curves, and the remaining one is isomorphic to $\Pb^1$ meeting each of the elliptic components at one node. 
Assume that two of the elliptic components are isomorphic, then there is  an involution of $C$ which permutes these two components and fixes the other two. In this case $C$ is a degeneration of billeliptic curves of genus three, and the involution described above is induced by the Prym involution on the curves degenerating to $C$. By construction the Prym variety of $C$ is isomorphic to a product of two elliptic curves. 
Assume moreover that the Prym varieties of the curves degenerating to $C$ admit a real multiplication by $\Oc_D$. Then so does the Prym variety of $C$. Thus this Abelian surface corresponds to a point in $P'_D$.

\medskip 

Consider an Abelian surface $A\simeq E_1\times E_2$ corresponding to a point in $P'_D$. 
A holomorphic $1$-form $\omega$ on $A$ is equivalent to a pair $((E_1,\omega_1),(E_2,\omega_2))$, where $\omega_i$ is a holomorphic $1$-form on $E_i$. By definition, $(E_i,\omega_i)$ is an element of the Hodge bundle $\Omega\Mcal_{1,1}$ over $\Mcal_{1,1}$.
It is a well known fact that $\Omega \Mcal_{1,1} \simeq \GL^+(2,\R)/\SL(2,\Z)$. Thus the group $\GL^+(2,\R)$ acts diagonally on $\Omega\Mcal_{1,1}\times \Omega\Mcal_{1,1}$.

By assumption,  $\End(A)$ contains a self-adjoint proper subring isomorphic to $\Oc_D$. Let $\Omega Q'_D$ be the spaces of pairs $(A,\omega)$, where $A\in P'_D$ and $\omega$ is an eigenform for the action of $\Oc_D$ on $\Omega^1(A)$. It is not difficult to see that $\Omega Q'_D$ consists of finitely many $\GL^+(2,\R)$-orbits in $\Omega\Mcal_{1,1}\times \Omega\Mcal_{1,1}$.    Let $Q'_D:=\Omega Q'_D/\C^*$. Then $Q'_D$ is a  finite cover of the modular curve $\Hbb/\SL(2,\Z)$. 

\medskip 

We have a natural map  $F: Q'_D \to P'_D$ which consists of forgetting the eigenform $\omega$ in the pairs $(A,[\omega])\in  Q'_D$.  Since for each $A\in P'_D$, we have $\Omega^1(A)= \C\cdot\omega\oplus\C\cdot\omega'$, where $\omega$ and $\omega'$ are two eigenvectors of $\Oc_D$, it follows that $F: Q'_D \to P'_D$ is a covering of degree two. To summarize, we have
 \begin{Proposition}\label{prop:QDp:n:PDp}
 	For all $D >9, D\equiv 0,1,4 \, [8]$, $D$ is not a square,  $Q'_D$ is a finite union of algebraic curves and we have
 	\begin{equation}\label{eq:eul:char:QDp:n:PDp}
 	\chi(P'_D)=\frac{1}{2}\cdot\chi(Q'_D).
 \end{equation}
 \end{Proposition}      

\subsection{Euler characteristics of the curves $Q'_D$}\label{subsec:eul:char:QDp}
Our goal now is to compute $\chi(Q'_D)$. 
A discriminant $D>9, \; D \equiv 0, 1, 4 \; [8]$ is said to be (1,2)-primitive if there does not $f\in \N, f >1$ such that $D=f^2D'$ with $D'\equiv 0,1,4 \; [8]$. We will show
\begin{Proposition}\label{prop:sum:eul:char:QDp}
 	Let $D_0 >9$ be a (1,2)-primitive discriminant. For any $f\in \N, \; f \geq 1$, we have
 	\begin{equation}\label{eq:sum:eul:char:Q:D:p:sigma:1}
 		\sum_{r \, | \, f}\chi(Q'_{r^2D_0}) =-\frac{1}{6}\cdot \sum_{\substack{ e^2 < f^2D_0 \\ e^2 \equiv f^2D_0 \; [8]}} \sigma_1\left(\frac{f^2D_0 -e^2}{8}\right).	
 	\end{equation} 
\end{Proposition} 
Let $\tilde{\pi}_{1,D}: \Omega Q'_D \to \Omega\Mcal_{1,1}$ be the map $((E_1,\omega_1), (E_2,\omega_2)) \mapsto (E_1,\omega_1)$. Denote by $\pi_{1,D}$ the induced map from $Q'_D$ to $\Mcal_{1,1}\simeq \Hbb/\SL(2,\Z)$.
\begin{Proposition}\label{prop:deg:proj:Q:D:p:to:mod:curv}
	Let $\Pc_D$ denote the sets of quadruples  of integers $(a,b,d,e)$ such that
	$$
	D=e^2+8ad, \; a>0, \; d>0, \; 0 \leq b < a, \; \gcd(a,b,d,e)=1.
	$$
	Then we have
	\begin{equation}\label{eq:deg:proj:Q:D:p:to:mod:curve}
	\deg \pi_{1,D}= \# \Pc_D.
\end{equation}
\end{Proposition}
 \begin{proof}[Sketch of proof]
 	Let $((E_1,\omega_1),(E_2,\omega_2))$ be an eigenform in $\Omega Q'_D$, where $(E_i,\omega_i) \in \Omega\Mcal_{1,1}$. 
 	Let $(\alpha_i,\beta_i)$ be a symplectic basis of $H_1(E_i,\Z), \; i=1,2$. Then $\Bc:=(\alpha_1,\beta_1,\alpha_2,\beta_2)$ is a basis of $H_1(E_1\times E_2,\Z)$. 
 	The polarization on $E_1\times E_2$ is associated with  the skew symmetric form $\langle.,.\rangle$ on $H_1(E_1\times E_2, \Z)$ given by the matrix $\left(\begin{smallmatrix} J & 0 \\ 0 & 2J \end{smallmatrix} \right)$, where $J=\left(\begin{smallmatrix} 0 & 1 \\ -1 & 0 \end{smallmatrix} \right)$, in the basis $\Bc$. 
 	
 	Let us write $\omega=\omega_1+\omega_2 \in \Omega^1(E_1\times E_2)$.
 	There is a unique element $T \in \End(E_1\times E_2)$ self-adjoint with respect to $\langle.,.\rangle$ which generates $\Oc_D$ and satisfies the following
 	\begin{itemize} 		
 	\item[$\bullet$] $T^*\omega=\lambda\omega$, where $\lambda \in \R_{>0}$,
 	
 	\item[$\bullet$] $T$ is given in the basis $\Bc$ by a matrix of the form
 	$$
 	T=\left(\begin{array}{cc}
 		eI_2 & 2B \\
 		B^* & 0_2
 	\end{array} \right)
 	$$
 	where $e\in \Z$, $B=\left(\begin{smallmatrix}
 		a & b \\ c & d
 	\end{smallmatrix} \right) \in \Mb_2(\Z)$, $B^*=\left(\begin{smallmatrix}
 	d & -b \\ -c & a
 	\end{smallmatrix} \right)$, and $\gcd(a,b,c,d,e)=1$. 
 	\end{itemize}
 	Note that the condition $\gcd(a,b,c,d,e)=1$ guarantees that the subring generated by $T$ is proper in $\End(E_1\times E_2)$.
 	
 	Since $T^2=eT+2\det B\cdot I_4$, the condition that $T$ generates $\Oc_D$ implies $e^2+8\det B =D$.
 	It is not difficult to see that $\det B >0$. Since $\lambda$ is the unique  positive root of the polynomial $x^2-ex-2\det B$, we must have $\lambda=(e+\sqrt{D})/2$.  
 	We can view $B$ as the matrix of a covering map $\varphi: E_2 \to E_1$. The condition $T^*\omega=\lambda\omega$ implies that 
 	\begin{equation}\label{eq:eig:form:product:cond}
 	\varphi^*\omega_1=\frac{\lambda}{2}\omega_2.
 	\end{equation}
 	There are lattices $\Lambda_1, \Lambda_2$ in $\C$ such that $(E_i,\omega_i) \simeq (\C/\Lambda_i,dz)$. Condition~\ref{eq:eig:form:product:cond} implies that $\frac{\lambda}{2}\cdot\Lambda_2$ is a sublattice of index $\det B$ in $\Lambda_1$. 
 	
 	Let us fix now a lattice $\Lambda_1\subset \C$, which represents an element of $\Mcal_{1,1}$. The degree of the map $\pi_{1,D}$ is the number of lattices $\Lambda_2$ such that $((\C/\Lambda_1,dz),(\C/\Lambda_2, dz)) \in \Omega Q'_D$. From the argument above this set is in bijection with the set of pairs $(e, \Lambda)$, where 
 	\begin{itemize}
 		\item[(i)] $e\in \Z, \; e^2 < D$, and $e^2 \equiv D \, [8]$.
 		
 		\item[(ii)] $\Lambda$ is a sublattice of index $(D-e^2)/8$ of $\Lambda_1$. 
 		\item[(iii)] If $B= \left(\begin{smallmatrix} a & b \\ c & d \end{smallmatrix} \right)\in \Mb_2(\Z)$ is the matrix of some basis of $\Lambda$ in a basis of  $\Lambda_1$,  then we must have $\gcd(a,b,c,d,e)=1$.   
 	\end{itemize} 
 	For simplicity,  assume that $\Lambda_1=\Z\oplus \imath \Z \simeq \Z^2$. Then every sublattice $\Lambda \subset \Z^2$ has a unique basis $(u,v)$ where 
 	\begin{itemize}
 		\item[$\bullet$] $u=(a,0)$ with $a>0$,
 		
 		\item[$\bullet$] $v=(b,d)$, where $0\leq b < a$, and $d>0$.
 	\end{itemize}  
 	Thus the set of $(e,\Lambda)$ satisfying the conditions (i), (ii), (iii) above is in bijection with the set
 	$$
 	\Pcal_D=\{(a,b,d,e) \in \Z^4, \, a >0, \, d > 0, \, 0 \leq b < a, \; D= e^2+8ad, \, \gcd(a,b,d,e)=1\},
 	$$
 	and the proposition follows.
 \end{proof}
 
\begin{Lemma}\label{lm:card:tPD:n:PD}
Let $D_0 >9$ be a (1,2)-primitive discriminant. For any $f\in \N, \; f \geq 1$,
let $\tilde{\Pc}_{f^2D_0}$ be the set of quadruples of integers $(a,b,d,e)$ such that 
$$
f^2D_0=e^2+8ad, \; a >0, \; d >0, \; 0 \leq b < a.
$$
Then we have	
\begin{equation}\label{eq:card:tPD:n:PD}
	\#\tilde{\Pc}_{f^2D_0} =\sum_{r \, | \, f}\#\Pc_{r^2D_0}.
\end{equation}
\end{Lemma}
 \begin{proof}
Consider $(a,b,d,e) \in \tilde{\Pc}_{f^2D_0}$. Let $s:=\gcd(a,b,d,e)$, and 
$$
	(a',b',d',e'):= \left(\frac{a}{s}, \frac{b}{s}, \frac{d}{s}, \frac{e}{s}\right).
	$$
	Since have $f^2D_0=e^2+8ad=s^2(e'{}^2+8a'd')$, it follows that $s^2 \, | \, f^2D_0$. We claim that $s \, | \, f$. Indeed, let $s':=s/\gcd(s,f)$ and $f'=f/\gcd(s,f)$. Then we have $s'{}^2 \, | \, f'{}^2D_0$, which implies that $s'{}^2 \, | \, D_0$. Assume that $s'>1$. If $s'$ is odd then 
	$$
	\frac{D_0}{s'{}^2} \equiv D_0 \; [8]
	$$ 
	which is a contradiction to the hypothesis that $D_0$ is $(1,2)$-primitive.
	Therefore $s'$ must be even, which implies that $f'$ is odd.
	But since
	$$
	f'{}^2\cdot\frac{D_0}{s'{}^2}=e'{}^2+8a'd'
	$$
	it follows that 
	$$
	\frac{D_0}{s'{}^2} \equiv e'{}^2+8a'd' \equiv 0,1,4 \; [8],
	$$
	and we have again a contradiction to the hypothesis that $D$ is $(1,2)$-primitive. Therefore we must have $s'=1$ that is $s \, | \, f$.
	Observe that we have $(a',b',d',e') \in \Pc_{(f/s)^2D_0}$.
	
	Conversely, for every $(a',b',d',e') \in \Pc_{r^2D_0}$, for some $r \, | \, f, \, r >0$. Let $s:= f/r$. Then $(sa',sb',sd',se')\in \tilde{\Pc}_{f^2D_0}$. We thus have a bijection between $\tilde{\Pc}_{f^2D_0}$ and the union $\bigsqcup_{r \, | \, f} \Pc_{r^2D_0}$, from which we get the desired conclusion. 
\end{proof}

\subsection*{Proof of Proposition~\ref{prop:sum:eul:char:QDp}}
\begin{proof}
	Let us write $D=f^2D_0$.  
	For each $e\in \Z$ such that $-\sqrt{D} < e < \sqrt{D}$ define
	$$
	\tilde{\Pc}_D(e):=\{(a,b,d)\in \N^3, \; (a,b,d,e) \in \tilde{\Pc}_D\}.
	$$
	One readily checks that $\#\tilde{\Pc}_D(e)=\sigma_1((D-e^2)/8)$.
	Thus 
	$$
	\#\tilde{\Pc}_{D} = \sum_{\substack{ -\sqrt{D} < e < \sqrt{D} \\ e^2 \equiv D \; [8]}}\sigma_1\left(\frac{D-e^2}{8}\right).
	$$
	On the other hand, from Proposition~\ref{prop:deg:proj:Q:D:p:to:mod:curv}, we have
	$$
	\chi(Q'_{r^2D_0})=\chi(\Hbb/\SL(2,\Z))\cdot \deg \pi_{1,r^2D_0} = \frac{-1}{6}\cdot\#\Pc_{r^2D_0}
	$$ 
	We then conclude by Lemma~\ref{lm:card:tPD:n:PD}.
\end{proof}

\section{Consequences of Bainbridge's formula}\label{sec:conseq:Bainbridge}
To prove our main results, it is essential to relate $\chi(X_{D})$ and $\chi(X_{D/4})$. To this purpose, we now prove

\begin{Proposition}\label{prop:conseq:Bainbridge}
	Let $D \in \N, \, D >9, \, D \equiv 0 \, [4]$, $D$ not a square. Assume that $D/4$ is a discriminant, that is $D/4 \equiv 0,1 \, [4]$. Then we have
	\begin{equation}\label{eq:ratio:H2D:n:H2D:4}
	\frac{\chi(X_D)}{\chi(X_{D/4})}= \left\{ 
	\begin{array}{cl}
		 6 & \hbox{ if $D/4 \equiv 1 \, [8]$} \\
		 10 & \hbox{ if $D/4 \equiv 5 \, [8]$} \\
		 8 & \hbox{ if $D/4 \equiv 0 \, [8]$}.  		 
	\end{array}
	\right.
	\end{equation} 
\end{Proposition}
\begin{proof}
	We can write $D=4^sf^2D_0$, where $s,f\in \N$, $s\geq 1$, $f$ odd, and $D_0$ is a fundamental discriminant. Recall from \eqref{eq:eul:char:XD} that we have
	$$
	\chi(X_{D/4})= 2\cdot(2^{s-1}f)^3\cdot \zeta_{K_{D_0}}(-1) \sum_{r|2^{s-1}f}\left( \frac{D_0}{r}\right)\cdot \frac{\mu(r)}{r^2}
	$$
	and
	$$
	\chi(X_D)= 2\cdot(2^sf)^3\cdot \zeta_{K_{D_0}}(-1)\left(\sum_{r|2 ^{s-1}f}\left( \frac{D_0}{r}\right)\cdot \frac{\mu(r)}{r^2} + \sum_{r | f } \left(\frac{D_0}{2^sr}\right)\cdot \frac{\mu(2^sr)}{4^sr^2}\right).
	$$
	\begin{itemize}
		\item[(a)] If $D/4 \equiv 1  \, [8]$ then $s=1$ and $D_0\equiv 1 \, [8]$. Note that we have $\left(\frac{D_0}{2}\right)=1$ in this case. Therefore,
		\begin{align*}
			\chi(X_D) & = 8\cdot 2\cdot f^3\cdot\zeta_{K_{D_0}}(-1)\cdot \sum_{r \, | \, f}\left(  \left(\frac{D_0}{r}\right)\cdot\frac{\mu(r)}{r^2} + \left(\frac{D_0}{2r}\right)\cdot\frac{\mu(2r)}{4r^2}
			\right)\\
			& =   8\cdot 2\cdot f^3\cdot\zeta_{K_{D_0}}(-1)\cdot \sum_{r \, | \, f} \left( \left(\frac{D_0}{r}\right)\cdot\frac{\mu(r)}{r^2} -\frac{1}{4}\left(\frac{D_0}{r}\right)\cdot\frac{\mu(r)}{r^2}
			\right)\\
			& = 6 \cdot 2\cdot f^3\cdot\zeta_{K_{D_0}}(-1)\cdot \sum_{r \, | \, f} \left(\frac{D_0}{r}\right)\cdot\frac{\mu(r)}{r^2} \\
			&= 6 \chi(X_{D/4}).
		\end{align*} 
	
	\item[(b)] If $D/4 \equiv 5  \, [8]$ then $s=1$ and $D_0\equiv 5 \, [8]$. Note that we have $\left(\frac{D_0}{2}\right)=-1$ in this case. Therefore,
	\begin{align*}
		\chi(X_D) & = 8\cdot 2\cdot f^3\cdot\zeta_{K_{D_0}}(-1)\cdot \sum_{r \, | \, f}\left(  \left(\frac{D_0}{r}\right)\cdot\frac{\mu(r)}{r^2} + \left(\frac{D_0}{2r}\right)\cdot\frac{\mu(2r)}{4r^2}
		\right)\\
		& =   8\cdot 2\cdot f^3\cdot\zeta_{K_{D_0}}(-1)\cdot \sum_{r \, | \, f} \left( \left(\frac{D_0}{r}\right)\cdot\frac{\mu(r)}{r^2} +\frac{1}{4}\left(\frac{D_0}{r}\right)\cdot\frac{\mu(r)}{r^2}
		\right)\\
		& = 10 \cdot 2\cdot f^3\cdot\zeta_{K_{D_0}}(-1)\cdot \sum_{r \, | \, f} \left(\frac{D_0}{r}\right)\cdot\frac{\mu(r)}{r^2} \\
		&= 10 \chi(X_{D/4}).
	\end{align*} 
	
	\item[(c)] Suppose that $D/4 \equiv 0 \, [4]$. We have two subcases
	   \begin{itemize}
	   	\item[(c.1)] $D_0$ is odd, that is $D_0 \equiv 1 \, [4]$. In this case, $s \geq 2$. Since $\mu(2^sr)=0$ if $s\geq 2$, it follows
	   	$$
	   	\chi(X_D)  = 8\cdot 2\cdot (2^{s-1}f)^3\cdot\zeta_{K_{D_0}}(-1)\cdot \sum_{r \, | \, f}\left(  \left(\frac{D_0}{r}\right)\cdot\frac{\mu(r)}{r^2} + \left(\frac{D_0}{2r}\right)\cdot\frac{\mu(2r)}{4r^2}
	   		\right) = 8\chi(X_{D/4}).
	   	$$
	   	
	   	\item[(c.2)] $D_0$ is even, that is $D_0 \equiv 0 \, [4]$. In this case, $\left( \frac{D_0}{r}\right)=0$ for all $r$ even. Thus we have 
	   	$$
	   	\chi(X_D)  = 8\cdot 2\cdot (2^{s-1}f)^3\cdot\zeta_{K_{D_0}}(-1)\cdot \sum_{r \, | \, f}\left(\frac{D_0}{r}\right)\cdot\frac{\mu(r)}{r^2} = 8\chi(X_{D/4}).
	   	$$
	   \end{itemize}
	\end{itemize}
\end{proof} 

Recall that for all discriminant $D>4$, $W_D(2)$ is the Teichm\"uller curve generated by Prym eigenforms in the locus $\Omega E_D(2) \subset \Omega \Mcal_2(2)$. Equivalently, $W_D(2)$ is the projection of $\Omega E_D(2)$ in $\Pb\Omega\Mcal_2(2)$. Note that $W_D(2)$ has two components if $D\equiv 1 \, [8]$ and one component otherwise (see~\cite{McM:spin}).
\begin{Corollary}\label{cor:eul:char:W2:D:W2:D:4:n:XDp}
	For all $D >0, \; 4 \, | \, D$, define 
	$$
	b_D=\left\{ 
	\begin{array}{cl}
		0 & \hbox{ if $D/4 \equiv 2,3 \, [4]$} \\
		4 & \hbox{ if $D/4 \equiv 0 \, [4]$} \\
		3 & \hbox{ if $D/4 \equiv 1 \, [8]$} \\
		5 & \hbox{ if $D/4 \equiv 5 \, [8]$} \\
	\end{array}
	\right.
	$$
	Then we have
	\begin{equation}\label{eq:eul:char:W2:D:W2:D:4:n:XDp}
	\chi(W_D(2))+b_D\chi(W_{D/4}(2))=-\frac{9}{2}\cdot\chi(X'_D).
	\end{equation}
\end{Corollary}
\begin{proof}
	It was shown in \cite{Bai:GT} that we have
	$$
	\chi(W_D(2))= -\frac{9}{2}\cdot\chi(X_D)
	$$
	Let us write $D=4^sf^2D_0$, where $s\in \N$, $f$ is odd, and $D_0$ is a fundamental discriminant.
	If $D/4 \equiv 2,3 \, [4]$ then $D/4$ is not a discriminant, which implies that $s=0$ and $4 \, | \, D_0$. Note that in this case $\chi(X_D)=\chi(X'_D)$ (cf. Theorem~\ref{prop:eul:char:XD:p}). Therefore, 
	$$
	\chi(W_D(2))=-\frac{9}{2}\cdot \chi(X_D) = -\frac{9}{2}\cdot \chi(X'_D).
	$$ 
	In the other cases $D/4$ is a discriminant, which implies that $s\geq 1$, and $\chi(X'_D)=\frac{3}{2}\cdot \chi(X_D)$. By Proposition~\ref{prop:conseq:Bainbridge}, we have
	\begin{align*}
		\chi(W_D(2))+b_D\chi(W_{D/4}(2)) & = -\frac{9}{2}\cdot\left(\chi(X_D)+b_D\chi(X_{D/4}) \right)\\
		& = -\frac{9}{2}\cdot\frac{3}{2}\cdot \chi(X_D)\\
		&=  -\frac{9}{2}\cdot\chi(X'_D).
	\end{align*}
\end{proof}

\section{Proof of Theorem~\ref{th:eul:char:PD:p:XD:p}}
By Proposition~\ref{prop:QDp:n:PDp}, \eqref{eq:eul:char:PD:p:XD:p} is equivalent to
\begin{equation}\label{eq:eul:char:QDp:n:XDp}
	\chi(Q'_D)= \left\{ 
	\begin{array}{cl}
		-2\chi(X'_D) &  \hbox{ if  $D\equiv 1 \, [8]$}\\
		-\chi(X'_D) & \hbox{ if $D\equiv 0 \, [4]$}.
	\end{array}	
	\right.
\end{equation}
In what follows, we will give the proof of \eqref{eq:eul:char:QDp:n:XDp}.

\subsection{Case $D \equiv 1 \, [8]$}
\begin{proof}
	In this case we can write $D=f^2D_0$, where $D_0\equiv 1 \, [8]$ is a fundamental discriminant. Note that $D_0$ is also $(1,2)$-primitive. Thus it follows from Proposition~\ref{prop:sum:eul:char:QDp} that we have
	$$
	-6\sum_{r \, | \, f} \chi(Q'_{r^2D_0})=\sum_{\substack{e^2 < D, \, e \, {\rm odd}}}\sigma_1\left(\frac{D-e^2}{8}\right).
	$$ 
	By Theorem~\ref{th:rel:sum:sigma:1}, we have
	$$
	\sum_{\substack{e^2 < D, \,  e \, {\rm odd}}}\sigma_1\left(\frac{D-e^2}{8}\right) = \frac{2}{5}\cdot \sum_{\substack{e^2 < D, \,  e  \, {\rm odd}}} \sigma_1\left(\frac{D-e^2}{4}\right).
	$$
	From Corollary~\ref{cor:sum:eul:char:XD}, we have
	$$
	30\sum_{r \, | \, f}\chi(X_{r^2D_0}) = \sum_{\substack{e^2 < D, \,  e \, {\rm odd}}}\sigma_1\left(\frac{D-e^2}{4}\right)
	$$
	which implies
	$$
	\sum_{r \, | \, f} \chi(Q'_{r^2D_0}) = -2\sum_{r \, | \, f}\chi(X_{r^2D_0}).
	$$
	By induction on $f$, we get the desired conclusion. 
\end{proof}

\subsection{Case $D\equiv 0 \, [4]$}
\begin{proof} 
Let us write  $D=4^sf^2D_0$, where $s\in \N$, $f$ is an odd number, and $D_0>1$ is a fundamental discriminant. 

\subsubsection{Case $s=0$} In this case $D=f^2D_0$, where $4 \, | \, D_0$, but $D_0/4$ is not a discriminant (that is $D_0/4 \equiv 2,3 \,[4]$). Thus we also have $D/4 \equiv 2,3 \, [4]$.  It follows from  Theorem~\ref{th:rel:sum:sigma:1} that we have
$$
\sum_{\substack{e^2 \equiv D \, [8]}}\sigma_1(\frac{D-e^2}{8}) = \frac{1}{5}\cdot \sum_{\substack{e^2 \equiv D \, [4]}}\sigma_1(\frac{D-e^2}{4})
$$
Since in this case $D_0$ is also $(1,2)$-primitive, by Proposition~\ref{prop:sum:eul:char:QDp}, we have
$$
-6\cdot \sum_{r \, | \, f}\chi(Q'_{r^2D_0}) = \sum_{\substack{e^2 \equiv D \, [8]}}\sigma_1\left( \frac{D-e^2}{8}\right)
$$
Combining with \eqref{eq:sum:eul:char:XD}, we get
$$
\sum_{r \, | \, f}\chi(Q'_{r^2D_0}) = -\sum_{r \, | \, f}\chi(X_{r^2D_0}).
$$
By induction on $f$, we conclude that 
$$
\chi(Q'_D)=-\chi(X_D)=-\chi(X'_D)
$$.

\subsubsection{Case $s=1$} In this case $D/4=f^2D_0$ is a discriminant. Thus Theorem~\ref{th:rel:sum:sigma:1} gives
$$
\sum_{\substack{e^2 \equiv D \, [8]}}\sigma_1(\frac{D-e^2}{8}) = \frac{1}{5}\cdot \sum_{\substack{e \, {\rm even}}}\sigma_1(\frac{D-e^2}{4}) + \frac{4}{5}\cdot \sum_{\substack{e^2 \equiv D/4 \, [4]}}\sigma_1(\frac{D/4-e^2}{4}). 
$$  
By Corollary~\ref{cor:sum:eul:char:XD}, we have
$$
\sum_{\substack{e \, {\rm even}}}\sigma_1(\frac{D-e^2}{4}) = 30\cdot \left(\sum_{r\, | \, f} \chi(X_{r^2D_0}) + \sum_{r\, | \, f} \chi(X_{4r^2D_0})\right) 
$$
and 
$$
\sum_{\substack{e^2 \equiv D/4 \, [4]}}\sigma_1(\frac{D/4-e^2}{4}) = 30\cdot \sum_{r\, | \, f} \chi(X_{r^2D_0}).
$$
Thus
$$
\sum_{\substack{e^2 \equiv D \, [8]}}\sigma_1(\frac{D-e^2}{8}) = 6\cdot\left(5\cdot\sum_{r\, | \, f} \chi(X_{r^2D_0}) + \sum_{r\, | \, f} \chi(X_{4r^2D_0})  \right). 
$$
\begin{itemize}
	\item[$\bullet$] If $D_0 \equiv 1 \, [8]$, then $D_0$ is $(1,2)$-primitive. It then follows from Proposition~\ref{prop:sum:eul:char:QDp} that we have
	$$
	-\left(\sum_{r \, | \, f}\chi(Q'_{r^2D_0})+ \sum_{r \, | \, f}\chi(Q'_{4r^2D_0}) \right) = 5\cdot\sum_{r\, | \, f} \chi(X_{r^2D_0}) + \sum_{r\, | \, f} \chi(X_{4r^2D_0}).
	$$
	From case $D \equiv  1 \, [8]$, we know that
	$$
	\sum_{r \, | \, f}\chi(Q'_{r^2D_0}) = -2\sum_{r\, | \, f} \chi(X_{r^2D_0}).
	$$
	Therefore
	$$
	-\sum_{r \, | \, f}\chi(Q'_{4r^2D_0})= 3\cdot\sum_{r\, | \, f} \chi(X_{r^2D_0}) + \sum_{r\, | \, f} \chi(X_{4r^2D_0})
	$$
	By Proposition~\ref{prop:conseq:Bainbridge} we have $\chi(X_{4r^2D_0}) = 6\chi(X_{r^2D_0})$ for all $r$ odd. Thus we have
	$$
	-\sum_{r \, | \, f}\chi(Q'_{4r^2D_0}) = \frac{3}{2}\cdot \sum_{r\, | \, f} \chi(X_{4r^2D_0}) =\sum_{r\, | \, f} \chi(X'_{4r^2D_0}).
	$$ 
	By induction on $f$, we conclude that $\chi(Q'_D)=-\chi(X'_D)$.
	
\item[$\bullet$] If $D_0 \equiv 0 \, [4]$, then $D_0$ is $(1,2)$-primitive. We also have  
$$
-\left(\sum_{r \, | \, f}\chi(Q'_{r^2D_0})+ \sum_{r \, | \, f}\chi(Q'_{4r^2D_0}) \right) = 5\cdot\sum_{r\, | \, f} \chi(X_{r^2D_0}) + \sum_{r\, | \, f} \chi(X_{4r^2D_0}).
$$
We have shown that $\chi(Q'_{r^2D_0}) = -\chi(X_{r^2D_0})$ for all $r$ odd (case $s=0$). Thus we obtain
$$
-\sum_{r \, | \, f}\chi(Q'_{4r^2D_0})= 4\cdot\sum_{r\, | \, f} \chi(X_{r^2D_0}) + \sum_{r\, | \, f} \chi(X_{4r^2D_0}).
$$	
In this case, we have $\chi(X_{4r^2D_0})=8\chi(X_{r^2D_0})$ by Proposition~\ref{prop:conseq:Bainbridge}. It follows that
$$
-\sum_{r \, | \, f}\chi(Q'_{4r^2D_0})= \frac{3}{2}\cdot\sum_{r\, | \, f} \chi(X_{4r^2D_0}) = \sum_{r\, | \, f} \chi(X'_{4r^2D_0})
$$
and therefore $\chi(Q'_D)=\chi(X'_D)$.

\item[$\bullet$] If $D_0 \equiv 5 \, [8]$ then $D_0$ is not $(1,2)$-primitive, but $4D_0$ is. Thus Proposition~\ref{prop:sum:eul:char:QDp} gives
$$
-\sum_{r \, | \, f}\chi(Q'_{4r^2D_0}) = 5\cdot\sum_{r\, | \, f} \chi(X_{r^2D_0}) + \sum_{r\, | \, f} \chi(X_{4r^2D_0}).
$$  
By Proposition~\ref{prop:conseq:Bainbridge}, $\chi(X_{4r^2D_0}) = 10 \chi(X_{r^2D_0})$ for all $r$ odd. Thus we have
$$
-\sum_{r \, | \, f}\chi(Q'_{4r^2D_0}) = \frac{3}{2}\cdot \sum_{r\, | \, f} \chi(X_{4r^2D_0}) =\sum_{r\, | \, f} \chi(X'_{4r^2D_0})
$$
which implies $\chi(Q'_D)=-\chi(X'_D)$.
\end{itemize}

\subsubsection{Case $s\geq 2$}
From Proposition~\ref{prop:sum:eul:char:QDp} we have
$$
-\sum_{i=0}^s\sum_{r \, | \, f} \chi(Q'_{4^ir^2D_0}) = 5\sum_{i=0}^{s-1}\sum_{r \, | \, f} \chi(X_{4^ir^2D_0}) + \sum_{r\, | \, f}\chi(X_{4^sr^2D_0})
$$
if $D_0 \equiv 0,1,4 \; [8]$, and
$$
-\sum_{i=1}^s\sum_{r \, | \, f} \chi(Q'_{4^ir^2D_0}) = 5\sum_{i=0}^{s-1}\sum_{r \, | \, f} \chi(X_{4^ir^2D_0}) + \sum_{r\, | \, f}\chi(X_{4^sr^2D_0})
$$
if $D_0 \equiv 5 \, [8]$. 
In both cases, by induction on $s$, we have
\[ 
-\sum_{r \, | \, f} \chi(Q'_{4^sr^2D_0})  = 4\sum_{r \, | \, f} \chi(X_{4^{s-1}r^2D_0}) + \sum_{r\, | \, f}\chi(X_{4^sr^2D_0}) 
 = \frac{3}{2}\cdot \sum_{r\, | \, f}\chi(X_{4^sr^2D_0})
 =  \sum_{r\, | \, f}\chi(X'_{4^sr^2D_0})
\]
which implies $\chi(Q'_D)= -\chi(X'_D)$. This completes the proof of Theorem~\ref{th:eul:char:PD:p:XD:p}.
\end{proof}

\section{Proof of Theorem~\ref{th:SV:const}}
\begin{proof}
	A triple of tori is the data of a tuple $(X,x_0,x_1,x_2,\omega)$, where
	\begin{itemize}
		\item[$\bullet$] $X$ is a disjoint union of three elliptic curves $X_0, X_1, X_2$.
		
		\item[$\bullet$] $x_j$ is a marked point in $X_j$.

		\item[$\bullet$] $\omega$ is a holomorphic Abelian differential on $X$, the restriction of $\omega$ to $X_j$ will be denoted by $\omega_j$. 
	\end{itemize}
	A triple of tori $(X,x_0,x_1,x_2,\omega)$ is called a {\em Prym form} if there exists an involution $\tau: X \to X$ which fixes $x_0$, permutes $x_1$ and $x_2$, and satisfies $\tau^*\omega=-\omega$. 
	Define $H_1(X,\Z)^-:=\{c\in H_1(X,\Z), \, \tau_*c = -c\}$. 
	Let $(\alpha_j,\beta_j)$ be a symplectic basis of $H_1(X_j,\Z), \; j=0,1,2$. We suppose that $\tau_*\alpha_1 =\alpha_2, \tau_*\beta_1=\beta_2$. Then $(\alpha_0,\beta_0,\alpha_1+\alpha_2,\beta_1+\beta_2)$ is a basis of $H_1(X,\Z)^-$. The restriction of the intersection form on $H_1(X,\Z)$ to $H_1(X,\Z)^-$ is given by the matrix $\left(\begin{smallmatrix}
		J & 0 \\ 0  & 2J 	\end{smallmatrix} \right)$, where $J=\left(\begin{smallmatrix} 0 & 1 \\ -1 & 0 \end{smallmatrix} \right)$.
	
	Let $\Omega(X)^-:=\ker(\tau^*+\id) \subset \Omega(X)$.
	The Prym variety of $(X,\tau)$ is defined to be
	$$
	\Prym(X,\tau):=(\Omega(X)^-)^*/H_1(X,\Z)^-.
	$$  
	It follows from the construction that $\Prym(X,\tau)$ is an Abelian surface whose polarization is of type $(1,2)$. Note that $\Prym(X,\tau)$ is in fact isomorphic to $X_0\times X_1$. 
	
	Since $\omega \in \Omega(X)^-$, we can view $\omega$ as an element of $\Omega^1(\Prym(X,\tau))$. By definition, $\Omega E_D(0^3)$ is the moduli space of triples of tori $(X,x_0,x_1,x_2,\omega)$ that are Prym forms such that $\Prym(X,\tau)$ admits a real multiplication by $\Oc_D$ for which $\omega$ is an eigenform. The curve $W_D(0^3)$ is the projectivization of $\Omega E_D(0^3)$ that is $W_D(0^3)=\Omega E_D(0^3)/\C^*$.
	The correspondence 
	$$
	(X,x_0,x_1,x_2,\omega) \mapsto (\Prym(X,\tau), \omega)
	$$
	provides us with a map from $\Omega E_D(0^3)$ to $\Omega Q'_D$, which is in fact an isomorphism. It follows that $W_D(0^3) \simeq Q'_D$. Hence
	$$
	\chi(W_D(0^3)) = \left\{ 
	\begin{array}{ll}
		-2\chi(X'_D) & \hbox{ if $D\equiv 1 \, [8]$} \\
		-\chi(X'_D) & \hbox{ if $4 \, | \, D$}.
	\end{array}
	\right.
	$$
	by Theorem~\ref{th:eul:char:PD:p:XD:p} and Proposition~\ref{prop:QDp:n:PDp}. By Corollary~\ref{cor:eul:char:W2:D:W2:D:4:n:XDp}, we have that 
	$$
	\chi(W_D(2))+b_D\chi(W_{D/4}(2))=-\frac{9}{2}\cdot\chi(X'_D)
	$$ 
	if $4 \, | \, D$. In case $D\equiv 1 \, [8]$, by \cite[Th. 1.1]{Bai:GT} and Proposition~\ref{prop:eul:char:XD:p}, we also have
	$$
	\chi(W_D(2)) = -\frac{9}{2}\chi(X_D)=-\frac{9}{2}\chi(X'_D). 
	$$
	Finally, from \cite[Th. 0.2]{Mo14}, we have
	$$
	\chi(W_D(4))=\left\{ 
	\begin{array}{ll}
		\frac{-5}{2}\cdot \chi(X'_D) & \hbox{ if $4 \, | \, D$} \\
		-5\cdot \chi(X'_D) & \hbox{ if $D \equiv 1 \, [8]$}
	\end{array} 
	\right. 
	$$
	Plugging those expressions of $\chi(W_D(4)), \chi(W_D(2)+b_D\chi(W_{D/4}(2)), \chi(W_D(0^3))$  to the formulas in Theorem~\ref{th:Siegel:Veech}, we get that
	$$
	c_1(D)=\frac{25}{9}, \quad c_2(D)=3, \quad c_3(D)=\frac{2}{9}  
	$$
	for all $D$. 
\end{proof}

\appendix
\section{Siegel-Veech constants of the locus $\tilde{\Qc}(4,-1^4)$}\label{sec:SV:quad:cov}
In this section, we give the  proof of the following
\begin{Theorem}\label{th:SV:constants:Prym:locus}
	Let $\tilde{\Qcal}(4,-1^4)$ denote the Prym locus in $\Omega \Mcal_3(2,2)^\odd$, that is the space of pairs $(X,\omega) \in \Omega \Mcal_3(2,2)^\odd$, where  $X$ admits an involution $\tau$ with four fixed points such that $\tau^*\omega=-\omega$. 
	For $k=1,2,3$, denote by $\tilde{c}_k(4,-1^4)$ denote the Siegel-Veech constant associated to saddle connections on $X$ joining the two zeros of the $\omega$ with multiplicity $k$. Then we have
	$$
	\tilde{c}_1(4,-1^4) = \frac{25}{9}, \quad  \tilde{c}_2(4,-1^4) = 3, \quad \tilde{c}_3(4,-1^4)=\frac{2}{9}.
	$$
\end{Theorem}
Let us first describe  briefly the strategy to compute  $\tilde{c}_k(4,-1^4)$ following Eskin-Masur-Zorich~\cite{EMZ03}.
Let $\Omega \Xcal$ be an affine invariant suborbifold in some stratum of translation surfaces. Consider a surface $(X,\omega) \in \Omega\Xcal$.
A {\em rigid family} of saddle connections on $(X,\omega)$ with respect to $\Omega\Xcal$ is a maximal collection of saddle connections  that have the same  periods on all surfaces in a neighborhood of $(X,\omega)$ in $\Omega\Xcal$.  We will be only  interested in saddle connections joining the two zeros of the $1$-forms in $\tilde{\Qcal}(4,-1^4)$. In this case a rigid family contains one, two, or  three saddle connections.

\medskip 

Let us fix $k\in \{1,2,3\}$. Let $\Omega \Mcal$ denote the space of Abelian differentials that are obtained by collapsing a rigid family of $k$ saddle connections (joining the two zeros of the $1$-form) on surfaces in $\tilde{\Qcal}(4,-1^4)$. 
For all $\delta\in\R_{>0}$ denote by $D'(\delta)$ the puncture disc $\{z\in \C, \; 0 < |z| < \delta\}$. 
For all $(Y,\eta) \in \Omega \Mcal$ there is  an  embedding $\Psi: U\times D'(\delta)/\Z_2 \to \tilde{\Qcal}(4,-1^4)$ for some $\delta>0$, where $U$ is a neighborhood of $(Y,\eta)$ in $\Omega\Mcal$, and the action of $\Z/2$ is generated by $(Y,\eta,z) \mapsto (Y,\eta,-z)$. 

Let $d\vol$ denote the Masur-Smillie-Veech measure on $\tilde{\Qcal}(4,-1^4)$. 
Then the pullback of $d\vol$  to $U\times D'(\delta)$ is decomposed as $d\vol'\otimes d\nu$, where 
\begin{itemize}
	\item[$\bullet$] $d\vol'$ is the restriction to $U$ of a volume form on $\Omega\Mcal$ that is proportional to the Lebesgue measure in local charts by period mappings,
	
	\item[$\bullet$] $d\nu$ is the pullback of the Lebesgue measure on $\C$ by a map of the form $z\to z^{\alpha}$. 
\end{itemize}
Let $\tilde{\Qcal}_1(4,-1^4)$ (resp. $\tilde{\Qcal}_{\leq 1}(4,-1^4)$)  denote the set of translation surfaces of unit area (resp. with area at most $1$) in $\tilde{\Qcal}(4,-1^4)$.  We define the subsets $\Omega_1\Mcal$ and $\Omega_{\leq 1}\Mcal$ of $\Omega\Mcal$ in the same way.
The volumes forms $d\vol$ and $d\vol'$ induce some volume forms $d\vol_1$ and $d\vol'_1$ on $\tilde{Q}_1(4,-1^4)$ and on $\Omega_1\Mcal$ respectively.
Note that $d\vol_1$ and $d\vol'_1$ are normalized  as follows
$$
\vol_1(\tilde{\Qcal}_1(4,-1^4))=\dim_\R \tilde{\Qcal}(4,-1^4)) \cdot \vol(\tilde{\Qcal}_{\leq 1}(4,-1^4))= 10\cdot \vol(\tilde{\Qcal}_{\leq 1}(4,-1^4))
$$ 
and
$$
\vol'_1(\Omega_1\Mcal)=\dim_\R \Omega\Mcal\cdot\vol'(\Omega_{\leq 1}\Mcal)=8\cdot \vol'(\Omega_{\leq 1}\Mcal). 
$$
It follows from the results of \cite{EMZ03} that we have
\begin{equation}\label{eq:SV:const:Q:4-1:4:k}
\tilde{c}_k(4,-1^4)=\frac{1}{2}\cdot \alpha \cdot\frac{\vol'_1(\Omega_1\Mcal)}{\vol_1(\tilde{\Qcal}_1(4,-1^4))}
\end{equation}
Note that the factor $\frac{1}{2}$ is introduced because the two zeros of the $1$-forms in $\tilde{\Qcal}(4,-1^4)$ cannot be distinguished. 
Since the preimages of the poles are not numbered in $\tilde{\Qcal}(4,-1^4)$ we have (cf. \cite[App. A]{Gouj:volumes})
$$
\vol_1(\tilde{\Qcal}_1(4,-1^4))=\frac{\vol_1(\Qcal_1(4,-1)^4)}{4!}= \frac{2\pi^4}{4!}=\frac{\pi^4}{12}.
$$

\subsection{Case $k=1$} In this case $\Omega \Mcal$ is the space of canonical double covers of quadratic differentials in the stratum $\Qcal(3,-1^3)$. It follows from \cite[\textsection 8B]{LN:finite} that $\alpha=5$, and in this case $d\vol'$ is equal to the Masur-Smillie-Veech volume on $\tilde{\Qcal}(3,-1^3)$. Taking into account the numbering of the zeros and poles in  $\Qcal(3,-1^3)$, we have
$$
\vol'_1(\tilde{\Qcal}_1(3,-1^3)) =\frac{5\pi^4}{9\cdot 3!} =\frac{5\pi^4}{54}.
$$
Thus we have
$$
\tilde{c}_1(4,-1^4)=\frac{1}{2}\cdot 5 \cdot \frac{5\pi^4/54}{\pi^4/12} =\frac{25}{9}.
$$

\subsection{Case $k=2$} In this case $\alpha=3$ and $\Omega \Mcal$ is the space of triples $(Y,\eta,w)$, where $(Y,\eta)$ is a translation of surface in $\Omega\Mcal_2(2)$ and $w$ is a Weierstrass point  on $Y$ which is not the zero of $\eta$ (see \cite[\textsection 8C]{LN:finite} for more details).  Let us denote this space by $\Omega \Mcal_2(2)^*$.  By definition $\Omega\Mcal_2(2)^*$ is a covering of degree $5$ over $\Omega \Mcal_2(2)$ with the covering map $(Y,\eta,w) \mapsto (Y,\eta)$. Denote by $d\vol^{(2)}$ the pullback to $\Omega \Mcal_2(2)^*$ of the Masur-Smillie-Veech volume on $\Omega\Mcal_2(2)$. We claim that 
\begin{Claim}\label{clm:ratio:dvols:on:H:2:s} 
We have
\begin{equation}\label{eq:ratio:dvols:on:H:2:s}
d\vol'=4d\vol^{(2)}.
\end{equation}
\end{Claim}
\begin{proof}
Consider a triple $(Y,\eta,w) \in \Omega\Mcal_2(2)^*$ which is obtained from a surface $(X,\omega) \in \tilde{\Qcal}(4,-1^4)$ by collapsing a pair of saddle connections $\{s_1,s_2\}$ that are exchanged by the Prym involution $\tau$. 
Recall that $\tilde{\Qcal}(4,-1^4)$ is locally modeled on $H^1(X,Z;\C)^-$, where $Z$ is the zero set of $\omega$. The volume form $d\vol$ is normalized such that the co-volume of the lattice 
$$
\{f \in H^1(X,Z;\C)^-, \; f(c)\in \Z+\imath\Z, \text{ for all } c \in H_1(X,Z;\Z)^-\}.
$$
is equal to $1$ (see \cite[\textsection 4.1]{AEZ16}). 
Let $(\alpha_1,\beta_1,\alpha_2,\beta_2)$ be a symplectic basis of $H_1(X,\Z)^-$ and $\gamma$ the element of $H_1(X,Z;\Z)$ represented by either $s_1$ or $s_2$. We can consider $\alpha_1,\beta_1,\alpha_2,\beta_2$, and $\gamma$ as $\C$-linear forms on $H^1(X,Z;\C)^-$. Denote by $\ol{\alpha}_1,\ol{\beta}_1,\ol{\alpha}_2,\ol{\beta}_2,\ol{\gamma}$ their complex conjugate. 
Then by definition, we have
\begin{equation}\label{eq:MSV:vol:on:tQ:4:-1:4}
	d\vol=\left(\frac{\imath}{2}\right)^5\alpha_1\wedge\ol{\alpha}_1\wedge\dots\wedge\beta_2\wedge\ol{\beta}_2\wedge\gamma\wedge\ol{\gamma}.
\end{equation}
on $H^1(X,Z;\C)^-$. 

By construction, the  involution $\tau$ on $X$ induces an involution $\iota$ on $Y$ which is in fact the hyperelliptic involution. 
There is an isometry from the complement of a neighborhood of $s_1\cup s_2$  in $X$ onto a complement of a neighborhood of $\{w,w_0\}$ in  $Y$, where $w_0$ is the unique zero of $\eta$. This map induces a morphism $\varphi: H_1(X,\Z)^- \to H_1(Y,\Z)^- = H_1(Y,\Z)$ (here, $H_1(Y,\Z)^-$ is the set $\{c\in H_1(Y,\Z), \; \iota_*c =-c\}$). One can choose a symplectic basis $(\alpha_1,\beta_1,\alpha_2,\beta_2)$ of $H_1(X,\Z)^-$ such that $(\varphi(\alpha_1)/2,\varphi(\beta_1), \varphi(\alpha_2),\varphi(\beta_2))$ is a symplectic basis of $H_1(Y,\Z)$ (see \cite[\textsection 8C]{LN:finite}).
Let us write $\alpha'_i=\varphi(\alpha_i), \beta'_i=\varphi(\beta_i)$. Then by construction, the volume form $d\vol'$ induced by $d\vol$ on $H^1(Y,\C)$ is given by
$$
d\vol' =  \left(\frac{\imath}{2}\right)^4 \alpha'_1\wedge\ol{\alpha}'_1\wedge\dots\wedge\beta'_2\wedge\ol{\beta}'_2.
$$
On the other hand the Masur-Smillie-Veech volume on $H^1(Y,\C)$ is given by
$$
d\vol^{(2)}=\left(\frac{\imath}{2}\right)^4\frac{\alpha'_1}{2}\wedge\frac{\ol{\alpha}'_1}{2}\wedge\dots\wedge\beta'_2\wedge\ol{\beta}'_2.
$$
Thus we have 
$$
d\vol'= 4 d\vol^{(2)}.
$$
\end{proof}
Now
$$
\vol'(\Omega_1\Mcal_2(2)^*) = 4\vol^{(2)}_1(\Omega_1\Mcal_2(2)^*) = 4\cdot 5 \cdot \vol^{(2)}_1 (\Omega_1\Mcal_2(2))=20\cdot\frac{\vol_1(\Qcal_1(1,-1^5))}{5!} = \frac{\pi^4}{6}.
$$
Therefore
$$
\tilde{c}_2(4,-1^4)=\frac{1}{2}\cdot 3 \cdot \frac{\pi^4/6}{\pi^4/12}=3.
$$
\subsection{Case $k=3$} In this case $\alpha=1$, and $\Omega \Mcal$ is the space $\Prym(0^3)$ of triples of tori $\{(Y_j,\eta_j,y_j), \, j=0,1,2\}$, where $(Y_1,\eta_1)$ and $(Y_2,\eta_2)$ are isomorphic. Let $Y$ denote the disjoint union of $Y_0,Y_1,Y_2$. The data of $\eta_0,\eta_1, \eta_2$ give a holomorphic $1$-form $\eta$ on $Y$. Thus we can  write $(Y,\eta,y_0,y_i,y_2)$ instead of $\{(Y_j,\eta_j,y_j), \; j=0,1,2\}$. 

Assume now that $(Y,\eta,y_0,y_1,y_2)\in \Prym(0^3)$ is obtained from $(X,\omega)\in \tilde{\Qcal}(4,-1^4)$ by collapsing a rigid family of three saddle connections $s_0,s_1,s_2$. Note that in this case $s_0, s_1, s_2$ are homologous  one to  the other. By construction, the Prym involution $\tau$ of $X$ induces an involution $\tau_Y$ on $Y$ which permutes $Y_1, Y_2$, and acts by $-\id$ on $H_1(Y_0,\Z)$. 
By construction, the $1$-form $\eta$ satisfies $\tau^*_Y\eta=-\eta$.
The space $\Prym(0^3)$ is thus  locally modeled on $H^1(Y,\C)^-:=\ker(\tau^*_Y+\id) \subset H^1(Y,\C)$.
Define 
$$
H_1(Y,\Z)^-=\{c\in H_1(Y,\Z), \; \tau_{Y*}c=-c\}.  
$$    
Let $(\alpha_j,\beta_j)$ be a symplectic basis of $H_1(Y_j,\Z)$ where $\tau_{Y*}\alpha_1=-\alpha_2, \tau_{Y*}\beta_1=-\beta_2$.   
Let $\alpha=\alpha_1+\alpha_2, \beta=\beta_1+\beta_2$.
Then $(\alpha_0,\beta_0,\alpha,\beta)$ is a symplectic basis of $H_1(Y,\Z)^-$.  
Consider $\alpha_0,\beta_0,\alpha,\beta$ as $\C$-linear forms on $H^1(Y,\C)^-$. Then the volume form $d\vol'$ on $\Prym(0^3)$ satisfies 
$$
d\vol'=\left(\frac{\imath}{2}\right)^4\alpha_0\wedge\ol{\alpha}_0\wedge\dots\wedge\beta\wedge\ol{\beta}.
$$
Our goal is to compute the volume of $\Prym_1(0^3):=\Omega_1 \Mcal^{(3)}$ with respect to $d\vol'_1$. To this purpose, we first remark that $\Prym(0^3)$ is isomorphic to $\Omega\Mcal_1(0)\times \Omega\Mcal_1(0)$ via the mapping 
$$
\Phi: (Y,\eta,y_0,y_1,y_2)\mapsto ((Y_0,\eta_0,y_0),(Y_1,\eta_1,y_1)).
$$ 
Since $\Omega\Mcal_1(0)\times \Omega\Mcal_1(0)$ is locally modeled on $H^1(Y_0,\C)\times H^1(Y_1,\C)$, for all $f\in H^1(Y,\C)^-$, we have 
$$
\Phi(f) = (f_{\left|H_1(Y_0,\C) \right.}, f_{\left|H_1(Y_1,\C)\right.})\in H^1(Y_0,\C)\times H^1(Y_1,\C)
$$ 
Let $d\vol^*$ denote the volume form on $H^1(Y_0,\C)\times H^1(Y_1,\C)$ which is defined by 
$$
d\vol^*=\left(\frac{\imath}{2}\right)^4\alpha_0\wedge \ol{\alpha}_0\wedge\beta_0\wedge\ol{\beta}_0\wedge\alpha_1\wedge\ol{\alpha}_1\wedge \beta_1\wedge\ol{\beta}_1.
$$
Via the map $\Phi$, we can consider $d\vol^*$ as a volume form on $H^1(Y,\C)^-$. Since for all $\xi\in H^1(Y,\C)^-$ we have $\xi(\alpha)=2\xi(\alpha_1), \xi(\beta)=2\xi(\beta_1)$, it follows that
$$
d\vol'=2^4\cdot d\vol^*.
$$ 
Note that we have
$$
\Aa(Y,\eta)=\Aa(Y_0,\eta_0)+2\Aa(Y_1,\eta_0).
$$
The desired conclusion follows from the following lemma.
\begin{Lemma}\label{lm:compare:vols:prod:sp}
Let $E$ be a $\C$-vector space admitting a splitting $E=V\times W$, where $\dim_\C V=r, \dim_\C W=s$. Assume that $V$ (resp. $W$) is endowed with a volume form $\mu$  and a Hermitian form $h$ (resp. a volume form $\nu$ and a Hermitian form $g$). Let $V^+:=\{v\in V, \; h(v,v) >0\}$ and $W^+=\{w\in W, \; g(w,w)>0\}$. Let $A$ and $B$ be some open cones in $V^+$ and in $W^+$ respectively. 
For all $t\in \R_{>0}$, define 
$$
A_t=\{v\in A, \;  h(v,v) = t\} \quad \text{and} \quad A_{<t}=\{ v \in A, \; 0 < h(v,v) < t\}.
$$
Similarly,
$$
B_t=\{w\in B, \; g(w,w)=t \} \quad \text{ and } \quad B_{<t} = \{ w \in B, \; 0< g(w,w) < t\}.
$$
Let $C:=A\times B \subset E$. Given  two positive real numbers $a,b$, define
$$
C(a,b)_{<1}:=\{(v,w) \in C, \; 0 < ah(v,v)+bg(w,w) <1\}
$$ 
Let $\lambda:=\mu\otimes \nu$.
Assume that $\mu(A_{<1})$ and $\nu(B_{<1})$ are finite.
Then we have
\begin{equation}\label{eq:vol:cone:prod}
\lambda(C(a,b)_{<1})=\frac{r!s!}{(r+s)!}\cdot\frac{\mu(A_{<1})\nu(B_{<1})}{a^rb^s}.
\end{equation}
\end{Lemma}
\begin{proof}
We can reparametrize $A$ by $\R_{>0}\times A_1$ via the map
$F(t,v) \mapsto tv$ for all $(t,v) \in \R_{>0}\times A_1$.
The pullback of $d\mu$ by $F$ can be written as $dt\wedge d\sigma_t$, where $\sigma_t$ is a volume form on $A_1$ depending on $t$. 
We claim that
\begin{equation}\label{eq:ratio:vols:At:A1} 
	d\sigma_t=t^{2r-1}d\sigma_1.
\end{equation}
Indeed, consider an open subset $\Omega$ of $A_1$ where a system of real local coordinates $(\omega_1,\dots, \omega_{2r-1})$ can be defined. Denote  the volume form $d\omega_1\wedge\dots\wedge d\omega_{2r-1}$ by $d\omega$.
Choose   a system of real coordinates $(v_1,\dots,v_{2r})$ on $V$ such that $d\mu=dv_1\wedge\dots\wedge dv_{2r}$. 
By definition,
$$
d\sigma_t = J_F(t,v)d\omega
$$
where $J_F$ is the Jacobian of $F$ with respect to the coordinate systems $(t,\omega_1,\dots,\omega_{2r-1})$ and $(v_1,\dots,v_{2r})$. 

Given $t\in \R_{>0}$, consider the map $\psi_t: V \to V, \; v\mapsto tv$. We have $\psi\circ F = F\circ L_t$, where $L_t(x,v)=(tx,v)$. By the chain rule, we get 
$$
d\psi_t(v)\circ dF(1,v) = dF(t,v)\circ dL_t(1,v).
$$
Since $	dL_t(1,v)=\left( 
	\begin{array}{cc}
		t & 0\\
		0 & I_{2r-1}
	\end{array}	
	\right)
	$ and $d\psi(v)=tI_{2r}$, we get that 
	$$
	t^{2r}J_F(1,v)=tJ_F(t,v)
	$$ 
	and \eqref{eq:ratio:vols:At:A1} follows. 	
As a consequence of \eqref{eq:ratio:vols:At:A1}, we get	
\begin{equation}\label{eq:ratio:vols:A1:n:CA1}
	\mu(A_{<1})=\int_0^1\left(\int_{A_1}d\sigma_t\right)dt =\int_{A_1}\sigma_1\cdot\int_0^1t^{2r-1}dt=\frac{\sigma_1(A_1)}{2r}.
\end{equation}
We now prove \eqref{eq:vol:cone:prod} in the case $a=b=1$.  By definition for all  $v\in A_{<1}$, we have $(v,w) \in C(1,1)_{<1}$ if and only if $w \in B_{<1-h(v,v)}$. Thus
	\begin{align*}
	\mu(C(1,1)_{<1}) & =\int_{A_{<1}} \nu(B_{< 1-h(v,v)})d\mu(v)=\nu(B_{<1})\cdot \int_{A_{<1}}(1-h(v,v))^sd\mu(v) \\
		 &=\nu(B_{<1})\sigma_1(A_1)\cdot\int_0^1(1-t^2)^st^{2r-1}dt\\
		 &=\nu(B_{<1})\cdot\frac{\sigma_1(A_1)}{2}\int_0^1(1-x)^sx^{r-1}dx \quad \hbox{(change of variable $x=t^2$)}\\
		 &= \nu(B_{<1})\cdot \frac{\sigma_1(A_1)}{2}\cdot\frac{s!(r-1)!}{(r+s)!}\\
		 & = \nu(B_{<1})\cdot \frac{\sigma_1(A_1)}{2r}\cdot\frac{s!r!}{(r+s)!}\\
		 & = \nu(B_{<1})\cdot \mu(A_{<1})\cdot\frac{s!r!}{(r+s)!} \quad \hbox{(by \eqref{eq:ratio:vols:A1:n:CA1})}.  	
	\end{align*}
For the general case, we replace $h$ by $h':=ah$ and $g$ by $g':=bg$. Since we have
$$
\mu(\{v\in A, \; 0< h'(v,v) < 1\}) =\mu(\{v \in A, \; 0 < h(v,v) < \frac{1}{a}\}) = \frac{\mu(A_{<1})}{a^r}
$$
and
$$
\nu(\{w\in B, \; 0< g'(w,w) < 1\}) =\nu(\{w \in B, \; 0 < g(w,w) < \frac{1}{b}\}) = \frac{\nu(B_{<1})}{b^s}
$$ 
equality \eqref{eq:vol:cone:prod} follows from the particular case $a=b=1$ applied to $h'$ and $g'$.	
\end{proof}
\begin{Remark}
	In the case $a=b=1$, Lemma~\ref{lm:compare:vols:prod:sp} is a reformulation a result in \cite[\textsection 6.2]{EMZ03}.
\end{Remark}

Applying Lemma~\ref{lm:compare:vols:prod:sp} to the case $E=H^1(Y,\C)^-$, $V=H^1(Y_0,\C), W=H^1(Y_1,\C)$, $h$ and $g$ correspond to area functions, $a=1$ and $b=2$, we obtain 
\begin{equation}\label{eq:vol:cones:prod:sp}
d\vol^*(\Prym_{\leq 1}(0^3))=\frac{2!2!}{2^2\cdot 4!}\cdot\mu(\Omega_{\leq 1}\Mcal_1(0))^2=\frac{\mu(\Omega_{\leq 1}\Mcal_1(0))^2}{24},
\end{equation}
where $\Prym_{\leq 1}(0^3)$ and $\Omega_{\leq 1}\Mcal_1(2)$ are the set of surfaces with area at most $1$ in $\Prym(0^3)$ and in $\Omega_{\leq 1}\Mcal_1(2)$ respectively, and $\mu$ is the Masur-Smillie-Veech volume on $\Omega \Mcal_1(0)$.

Let $\vol^*_1$ and $\mu_1$ be the measures induced by $\vol^*$ and by $\mu$ on $\Prym_1(0^3)$ and on $\Omega_1\Mcal_1(0)$ respectively. By definition, we have
$$
\vol^*_1(\Prym_1(0^3))=8\vol^*(\Prym_{\leq 1}(0^3)) \quad \text{ and } \quad \mu_1(\Omega_1\Mcal_1(2))=4\mu(\Omega_{\leq 1}\Mcal_1(0)).
$$ 
Thus we get
$$
\vol^*_1(\Prym_1(0^3))=\frac{\mu_1(\Omega_1\Mcal_1(0))^2}{48}.
$$ 
Now as $\mu_1(\Omega_1\Mcal_1(0))=\pi^2/3$ (see \cite[Tab. 1]{EMZ03}), and $d\vol'=16d\vol^*$, we have
$$
\displaystyle \tilde{c}_3(4,-1^4)=\frac{1}{2}\cdot 1 \cdot\frac{16\pi^4/(9\cdot 48)}{\pi^4/12}=\frac{2}{9}.
$$

\end{document}